\documentclass[12pt,reqno]{amsart}
\usepackage{graphicx} 
\usepackage{amssymb,amscd, amsthm, latexsym, mathrsfs, epsfig, appendix}
\usepackage[all]{xy}
\usepackage[shortlabels]{enumitem}

\newcommand\A{\mathcal{A}}
\newcommand\C{\mathbb C}

\newcommand\Q{\mathbb Q}
\newcommand\R{\mathbb R}
\newcommand\Z{\mathbb Z}

\newcommand{\G}{\mathcal{G}}
\renewcommand{\H}{\mathcal{H}}

\newcommand{\re}{\operatorname{Re}}
\newcommand{\im}{\operatorname{Im}}

\newcommand{\hth}{\hat{\theta}}
\newcommand{\ii}{\sqrt{-1}}

\newcommand\Aut{\operatorname{Aut}}

\newcommand{\del}{\partial}
\newcommand{\delbar}{\overline{\partial}}
\newcommand{\Vol}{\operatorname{Vol}}

\newcommand{\Hes}{\operatorname{Hess}}

\newcommand{\tr}{\operatorname{tr}}

\makeatletter \@addtoreset{equation}{section} \makeatother

\newtheorem{thm}{Theorem}
\newtheorem{prop}[thm]{Proposition}
\newtheorem{lem}[thm]{Lemma}
\newtheorem{cor}[thm]{Corollary}

\theoremstyle{definition}

\newtheorem{rmk}[thm]{Remark}

\newcommand\narrowdots{\hbox to 1em{.\hss.\hss.}}

\title[$\textrm{d}$HYM connections and scalar curvature]{Deformed Hermitian Yang-Mills connections, extended gauge group\\ and scalar curvature}
\author{Enrico Schlitzer}
\email{eschlitz@sissa.it}
\author{Jacopo Stoppa}
\email{jstoppa@sissa.it}
\address{SISSA, via Bonomea 265, 34136 Trieste, Italy}
\address{IGAP, Via Beirut 2, 34151 Trieste, Italy}
\begin{document}
\begin{abstract} The deformed Hermitian Yang-Mills (dHYM) equation is a special Lagrangian type condition in complex geometry. It requires the complex analogue of the Lagrangian phase, defined for Chern connections on holomorphic line bundles using a background K\"ahler metric, to be constant. In this paper we introduce and study dHYM equations with variable K\"ahler metric. These are coupled equations involving both the Lagrangian phase and the radius function, at the same time. They are obtained by using the extended gauge group to couple the moment map interpretation of dHYM connections, due to Collins-Yau and mirror to Thomas' moment map for special Lagrangians, to the Donald\-son-Fujiki picture of scalar curvature as a moment map. As a consequence one expects that solutions should satisfy a mixture of K-stability and Bridgeland-type stability. In special limits, or in special cases, we recover the K\"ahler-Yang-Mills system of \'Alvarez-C\'onsul, Garcia-Fernandez and Garc\'ia-Prada, and the coupled K\"ahler-Einstein equations of Hultgren-Witt Nystr\"om. After establishing several general results we focus on the equations and their large/small radius limits on abelian varieties, with a source term, following ideas of Feng and Sz\'ekelyhidi. 
\end{abstract}

\maketitle
\section{Introduction} 
Let $X$ be a compact $n$-dimensional K\"ahler manifold, with a fixed K\"ahler form $\omega$ and a holomorphic line bundle $L \to X$. For all Hermitian metrics $h$ on the fibres of $L$, the integral $\int_X (\omega - F(h))^n$ (where $F(h) \in \mathcal{A}^{1,1}(X, \ii \R)$ denotes the curvature of the Chern connection of $h$) only depends on $[\omega]$, $c_1(L)$. Assuming that it does not vanish, one defines uniquely a phase $e^{\sqrt{-1}\hth} \in U(1)$ by requiring 
\begin{equation*}
\int_X (\omega - F(h))^n \in \R_{> 0} e^{\ii \hth}. 
\end{equation*} 
Then we say that $h$ satisfies the \emph{deformed Hermitian Yang-Mills (dHYM) equation}\footnote{We write the equation following the sign convention of Collins-Yau \cite{collinsYau}, Section 2. It is important to note that in most references \eqref{dHYM} corresponds to the dHYM equation for the inverse line bundle $L^{-1}$.} if 
\begin{equation}{\label{dHYM}}
\im\left(e^{-\ii\hth} (\omega - F(h))^n\right) = 0.
\end{equation}
The dHYM equation (which first appeared in the physics literature \cite{marino}) has attracted considerable recent interest as the $B$-model analogue of the $A$-model special Lagrangian condition in homological mirror symmetry (see e.g. \cite{collinsYau, collinsXie, jacob, leungYau}). In this context one thinks of $X$ as a holomorphic submanifold (not necessarily proper) of an ambient Calabi-Yau $M$, with a mirror $\check{M}$. A remarkable conjectural consequence of this picture is that a solution of the dHYM equation should exist if and only if $L$ satisfies a suitable Bridgeland-type stability condition induced by $\omega$. Collins, Jacob and Yau \cite{collinsJacobYau, collinsYau} prove precise results in this direction, which do not depend on the ambient $M$. In particular they show that if $L$ admits a solution of the dHYM equation satisfying a suitable (``hypercritical phase") condition, and if one defines 
\begin{equation*}
Z_V(L) = -\int_V e^{-\ii \omega} ch(L)
\end{equation*} 
for all analytic subvarieties $V \subset X$, then $Z_V(L)$ always lies in the upper half-plane and, when $V \subset X$ is a proper subvariety, one has  
\begin{equation}\label{Bstab}
\operatorname{Arg} Z_V(L) > \operatorname{Arg} Z_X(L).
\end{equation}

In practice, the dHYM equation is usually written in terms of the Lagrangian phase operator, see e.g. \cite{collinsYau} Section 2. Namely one introduces the endomorphism of the tangent bundle given by $\omega^{-1}\left(\ii F(h)\right)$, with (globally well-defined) eigenvalues $\lambda_i$. Then we have
\begin{align*}
\frac{(\omega - F(h))^n}{\omega^n} &= \prod^n_{i=1}(1+\ii\lambda_i)\\
&= r_{\omega}(h) e^{\ii\Theta_{\omega}(h)},
\end{align*}
where the \emph{radius function} and \emph{Lagrangian phase operator} are defined (using the background metric $\omega$) respectively as
\begin{align*}
&r_{\omega}(h) = r_{\omega}\left(\ii F(h)\right) = \prod^n_{i=1}\sqrt{1+\lambda^2_i},\\
&\Theta_{\omega}(h)=\Theta_{\omega}\left(\ii F(h)\right) = \sum^n_{i=1}\arctan(\lambda_i).
\end{align*}
The dHYM equation becomes
\begin{equation*}
\Theta_{\omega}(h) = \hth \mod 2\pi.
\end{equation*}
\vskip.5cm
In the present paper we introduce and study dHYM equations with variable K\"ahler metric. These involve both the Lagrangian phase and the radius function, at the same time. The basic idea is that the radius function should be used to couple the dHYM metric $h$ to a variable K\"ahler metric $\omega$. The equations we propose are
\begin{align}\label{coupled_dHYM_Intro}
\begin{cases}
\Theta_{\omega}(h) = \hth \mod 2\pi\,\\
s(\omega) - \alpha r_{\omega}(h) = \hat{s} - \alpha \hat{r}, 
\end{cases} 
\end{align}
where $s(\omega)$, $\hat{s}$, $\hat{r}$ denote the scalar curvature and its average, respectively the average radius, and $\alpha > 0$ is an arbitrary coupling constant. The quantities $\hat{s}$, $\hat{r}$ are fixed by cohomology, and in particular 
\begin{equation*}
\hat{r} = \frac{1}{n!\Vol(M,\omega)}\left|\int_{X} (\omega- F(h))^n\right|. 
\end{equation*} 
Our coupled equations \eqref{coupled_dHYM_Intro} are natural as they are obtained by lifting the moment map interpretation of dHYM connections, due to Collins-Yau \cite{collinsYau} and mirror to Thomas' moment map for special Lagrangians \cite{richard}, to the extended gauge group of bundle automorphisms covering Hamiltonian symplectomorphisms (see Theorem \ref{ExtendedGaugeThm}). The well-known Donald\-son-Fujiki picture of scalar curvature as a moment map then allows coupling to the underlying K\"ahler metric, through the scalar (or, in special cases, the Ricci) curvature (see Corollary \ref{ExtendedGaugeCor} and Section \ref{RicciSubSec}). The resulting moment map partial differential equations \eqref{coupled_dHYM_Intro} describe special pairs formed by a holomorphic line bundle $L$, regarded as a $B$-model object, and a K\"ahler class, and it is natural to expect that these should satisfy a mixture of Bridgeland-type stability (as in \cite{collinsYau}, see e.g. the inequality \eqref{Bstab}), and K-stability (see e. g. \cite{donaldson}). In special limits, or in special cases, we recover the very interesting systems introduced by \'Alvarez-C\'onsul, Garcia-Fernandez and Garc\'ia-Prada \cite{AGG} (in the large radius limit, see Proposition \ref{LargeRadiusProp}) and Hultgren-Witt Nystr\"om \cite{wytt} (on complex surfaces, see Section \ref{RicciSubSec} and in particular Corollary \ref{CoupledKECor}). 
\begin{rmk} One can allow a general class in $H^{1,1}(X, \R)$ in the dHYM equation \eqref{dHYM}, not necessarily the first Chern class of a line bundle. At least in the absence of holomorphic $2$-forms, this can be interpreted as allowing a nontrivial $B$-field, as discussed in \cite{collinsYau} Section 8. Thus our coupled equations \eqref{coupled_dHYM_Intro} also admit a different (although closely related) interpretation: we may replace $[\ii F(h)]$ with a minimal lift of a $B$-field class $[B] \in H^2(X, \R)/H^2(X, \Z)$ (which exists under suitable assumptions), and regard the equations as trying to prescribe a \emph{canonical representative} $\omega + \ii B$ of a complefixied K\"ahler class $[\omega] + \ii [B]$, much as the cscK equation $s(\omega) = \hat{s}$ tries to find a canonical representative for $[\omega]$. Note that in the Calabi-Yau case, at zero coupling $\alpha = 0$ and in the large radius limit discussed in Proposition \ref{LargeRadiusProp}, these equations for $\omega + \ii B$ reduce to the conditions 
\begin{align*} 
\begin{cases}
\Delta_{\omega} B = 0\\
\operatorname{Ric}(\omega) = 0, 
\end{cases} 
\end{align*}
which are standard in the physics literature (see e.g. \cite{branes} Section 1.1).
\end{rmk}
We will study the coupled equations \eqref{coupled_dHYM_Intro} from the moment map point of view. After establishing several general results, in order to analyse a concrete case in more detail, we focus on the equations and their large/small radius limits on abelian varieties, with a source term, following ideas of Feng and Sz\'ekelyhidi \cite{fengSz}. In particular we prove a priori estimates (see Propositions \ref{aprioriPropSurfaces}, \ref{aprioriProp}) from which we can deduce existence in some cases (see Theorems \ref{MainThmSurface}, \ref{MainThmLargeRadius}, \ref{MainThmODE}, \ref{LargeSmallRadiusThmODE}). Our main results, together with the necessary background, are contained in Section \ref{BackSec}.\\ 
\noindent{\textbf{Acknowledgements.}} We are very grateful to Tristan Collins, Mario Garcia-Fernandez and to an anonymous Referee for important corrections and suggestions on this manuscript. We also thank Alessandro Tomasiello and all the participants in the K\"ahler geometry seminar at IGAP, Trieste.  
\section{Background and main results}\label{BackSec}
\subsection{Moment map description} Thomas \cite{richard} gave a moment map interpretation of the special Lagrangian equation on submanifolds of $\check{M}$. Mirror to this, there is a moment map description of the dHYM equation \eqref{dHYM}, due to Collins and Yau, which in fact is intrinsic to $X$. Namely, fixing a metric $h$ as above, one considers the space $\mathscr{A}^{1,1}$ of $h$-unitary integrable connections on $L$, endowed with the natural action of the gauge group $\G$ of unitary bundle automorphisms of $L$ covering the identity on $X$, and the (nonstandard, nonlinear, possibly degenerate\footnote{According to \cite{collinsYau} Section 2, it is nondegenerate at least in an open neighbourhood of a solution to the dHYM equation.}) symplectic form given at $A \in \mathscr{A}^{1,1}$ by 
\begin{equation*}
\omega^{\text{dHYM}}_A( a, b) = -\int_X a\wedge b\wedge\re\left(e^{-\ii\hth} (\omega - F(A))^{n-1}\right),
\end{equation*} 
with $a,b \in T_A \mathscr{A}^{1,1} \subset \A^1(X,\sqrt{-1}\mathbb{R})$. According to \cite{collinsYau} Section 2 the action of $\G$ on $\mathscr{A}^{1,1}$ is Hamiltonian, with equivariant moment map at $A \in \mathscr{A}^{1,1}$, evaluated on $\varphi \in \operatorname{Lie}(\G)$, given by
\begin{equation}\label{CollinsMomMap} 
\langle\mu_{\G}(A), \varphi\rangle = \frac{\ii}{n}\int_X \varphi\im  \left(e^{-\ii\hth} (\omega - F(A))^n\right).
\end{equation}
A standard argument then shows that the dHYM equation \eqref{dHYM} becomes precisely the problem of finding zeroes of the moment map $\mu_{\G}$ inside the orbits of the complexified gauge group $\G^{\C}$.
\subsection{Lift to the extended gauge group} Fix a Hermitian metric $h$ on $L$ as above. The \emph{extended gauge group} $\widetilde{\G}$ of $L$ consists of unitary bundle automorphisms of $(L, h)$ which cover a Hamiltonian symplectomorphism of $X$, with respect to the fixed symplectic (in fact K\"ahler) form $\omega$. It can be shown that $\widetilde{\G}$ fits into an exact sequence of infinite-dimensional Lie groups
\begin{equation}\label{LieSplitting}
1 \to \mathcal{G} \xrightarrow{\,\,\,\iota\,\,\,} \widetilde{\G} \xrightarrow{\,\,\,p\,\,\,} \H \to 1,
\end{equation}
where $\H$ denotes the group of Hamiltonian symplectomorphisms of $(X, \omega)$. 

As observed by \'Alvarez-C\'onsul, Garcia-Fernandez and Garc\'ia-Prada \cite{AGG}, there is a natural action of $\widetilde{\G}$ on the space of all unitary connections $\mathscr{A}$, given by thinking of a connection $A$ as a projection operator $\theta_A$ on the vertical bundle,
\begin{equation}\label{VertProj}
\theta_A\!: TL \to VL,\quad g\cdot \theta_A = g_* \circ \theta_A \circ (g_*)^{-1}.
\end{equation}
The resulting action was studied in detail in \cite{AGG}, in the much more general context of arbitrary principal bundles for a compact real Lie group. The main application considered in \cite{AGG} concerns the case when the space of connections is endowed with the standard, linear Atiyah-Bott symplectic form, which for a line bundle is
\begin{equation*}
\omega^{\text{AB}}_A( a, b) = -\int_X a\wedge b\wedge \omega^{n-1}.
\end{equation*}
However we can use this general setup to obtain a result for the symplectic form $\omega^{\rm{dHYM}}$.
\begin{thm}\label{ExtendedGaugeThm} The action of the extended gauge group $\widetilde{\G}$ on the space of all unitary connections $\mathscr{A}$, endowed with obvious extension of the symplectic form $\omega^{\rm{dHYM}}$, is Hamiltonian, with equivariant moment map at $A \in \mathscr{A}$, evaluated on $\zeta\in\operatorname{Lie}(\widetilde{\G})$, given by 
\begin{equation*}
\langle \mu_{\widetilde{\G}}(A), \zeta \rangle = \langle \mu_{ \G }(A), \theta_A(\zeta) \rangle +\frac{1}{n}\int_X p_*(\zeta) \re  \left(e^{-\ii\hth} (\omega - F(A))^n\right),    
\end{equation*}
where $\mu_{ \G }$ is defined as in \eqref{CollinsMomMap} and $p$ is the projection appearing in \eqref{LieSplitting}.
\end{thm}
Notice that we are identifying the Lie algebra $\operatorname{Lie}({\H})$ of Hamiltonian vector fields on $(X,\omega)$ with the space $C^{\infty}_0(X,\omega)$ of Hamiltonian functions with zero average.
Theorem \ref{ExtendedGaugeThm} is proved in Section \ref{gaugeSec}. Its analogue for the standard symplectic form $\omega^{\rm{AB}}$ is (a special case of) \cite{AGG}, Proposition 1.6: we will return to this below.
\subsection{Coupling to a variable metric} It was shown by Jacob and Yau \cite{jacob} that solutions of the dHYM equation are unique. It follows that the naive problem of looking for zeroes of the extended moment map $\mu_{\widetilde{\G}}$ is overdetermined.

The fundamental work of Donaldson \cite{donaldson} and Fujiki \cite{fujiki} on scalar curvature as a moment map suggests that the right thing to do instead is to let the extended gauge group act on a larger space. Let $\mathscr{A}$ be the space of all unitary connections on $(L, h)$, as above, and let $\mathscr{J}$ be the space of $\omega$-compatible almost complex structures on $X$. We endow $\mathscr{A}$ with the symplectic form given by the obvious extension of $\omega^{\rm{dHYM}}$ and $\mathscr{J}$ with the Donaldson-Fujiki form $\omega^{\rm{DF}}$. We consider the induced action of the extended gauge group $\widetilde{\G}$ on the product $\mathscr{A}\times\mathscr{J}$, which on the second factor is given by 
\begin{equation*}
g \cdot J = p(g)_* \circ J \circ p(g)_*^{-1}.
\end{equation*}
This preserves the space $\mathscr{P} \subset \mathscr{A}\times\mathscr{J}$ consisting of pairs $(A, J)$ of a unitary connection $A$ and an integrable complex structure $J$, such that $A$ is integrable with respect to $J$. We denote by $s(J)$ the scalar curvature of the metric determined by $\omega$ and $J$, and let $\alpha$ be a real \emph{positive} ``coupling constant". 
\begin{cor}\label{ExtendedGaugeCor} The action of $\widetilde{\G}$ on $\mathscr{A}\times\mathscr{J}$, endowed with the symplectic form 
\begin{equation*}
\omega_{\alpha} = n\alpha\omega^{\rm{dHYM}} + \,\omega^{\rm{DF}},
\end{equation*}
for $\alpha > 0$, is Hamiltonian, with equivariant moment map at $(A, J)$, acting on $\zeta \in \operatorname{Lie}(\widetilde{\G})$, given by \begin{equation*}
\langle \mu_{\alpha}(A, J), \zeta \rangle = -\int_X p_*(\zeta) s(J) \frac{\omega^n}{n!} + n\alpha \langle \mu_{\widetilde{\G}}(A), \zeta \rangle.    
\end{equation*}
\end{cor}  
The proof is given in Section \ref{gaugeSec}. The analogue of this for the Atiyah-Bott symplectic form is \cite{AGG}, Proposition 2.1. 

For each $A \in \mathscr{A}$, the operator $\theta_A$ induces an equivariant \emph{vector space} splitting   
\begin{equation*}
\operatorname{Lie}(\widetilde{\G}) = \operatorname{Lie}(\G) \bigoplus_{\theta_A} \operatorname{Lie}(\H).
\end{equation*}
We consider the problem of finding a pair $(A, J) \in \mathscr{P}$ such that, in the above splitting with respect to $\theta_A$, the moment map $\mu_{\alpha}$ vanishes on $\operatorname{Lie}(\G)$ and acts as some prescribed element $-f \in C^{\infty}_0(X, \omega) \cong \operatorname{Lie}(\H)^*$ on Hamiltonian vector fields:
\begin{equation*}
\langle \mu_{\alpha}(A, J), \zeta  \rangle = \langle \mu_{\alpha}(A, J), \zeta'\oplus_{\theta_A}\zeta'' \rangle = -\int_X \zeta'' f \frac{\omega^n}{n!}.    
\end{equation*}
This is equivalent to the equations
\begin{align*}
\begin{cases}
\im\left(e^{-\ii\hth} (\omega - F(A))^n\right)=0\\
s(J) - \alpha \frac{\re\left(e^{-\ii\hth} (\omega - F(A))^n\right)}{\omega^n} = f + c, 
\end{cases}
\end{align*} 
where the only possible constant $c$ is determined once the coupling $\alpha$ is chosen, depending only on $[\omega]$ and $c_1(L)$. Note that the problem of finding a zero of the moment map $\mu_{\alpha}$ corresponds to the choice $f \equiv 0$.

Formally complexifying the action of $\widetilde{\G}$, following Donaldson \cite{donaldson}, we keep the complex structures on $L$ and $X$ fixed and vary the Hermitian metric $h$ on $L$ and the K\"ahler form $\omega$ in its K\"ahler class instead. Thus we arrive at the \emph{dHYM equation coupled to a variable K\"ahler metric} 
\begin{align}\label{coupled_dHYM}
\begin{cases}
\im\left(e^{-\ii\hth} (\omega - F(h))^n\right)=0\\
s(\omega) - \alpha \frac{\re\left(e^{-\ii\hth} (\omega - F(h))^n\right)}{\omega^n} = f, 
\end{cases} 
\end{align}
where we absorbed the constant $c$ in the datum $f$. 

These equations are the main object of study in the present paper. Important motivation for this study comes from the fact that, when the datum $f$ is constant, so that we are looking for zeroes of the moment map, it is natural to expect that solutions should satisfy a mixture of Bridgeland-type stability and K-stability. Note that it is straightforward to rewrite \eqref{coupled_dHYM} in terms of the Lagrangian phase operator and the radius function, as we did in \eqref{coupled_dHYM_Intro}.
 
\begin{rmk}\label{realizRmk} Given the origin of the dHYM equation, it is natural to ask whether a given solution of the coupled equations \eqref{coupled_dHYM} can be effectively realised in the $B$-model, that is if the pair $(X, \omega)$ underlying a solution can be embedded holomorphically and isometrically in a Calabi-Yau $M$. Note that, at the infinitesimal level, this is always possible: for example, by the results of \cite{kaledin}, we may embed $(X, \omega)$ isometrically as the zero section of the holomorphic cotangent bundle $T^*X$, endowed with a hyperk\"ahler metric defined in a \emph{formal} neighbourhood of the zero section.
\end{rmk}
\begin{rmk}\label{scalingRmk} As we mentioned in the Introduction, the dHYM equation makes sense and plays a role in mirror symmetry even when we replace $c_1(L)$ with some arbitrary class $[F]$ of type $(1,1)$, not necessarily rational. The same holds for the coupled equations \eqref{coupled_dHYM}. When $[F]$ is rational, we can obtain solutions in a class $c_1(L)$ from solutions in $[F]$ by rescaling $F$ and $\omega$ appropriately. We will often use this fact, sometimes without further comment. 
\end{rmk}  
\subsection{Reduction to Ricci curvature}\label{RicciSubSec} We describe a special case in which the dependence on the scalar curvature in the equations \eqref{coupled_dHYM} can be reduced to the Ricci curvature. To see this we note that by using the dHYM equation in order to eliminate the top power $(F(h))^n$ we can always write 
\begin{align*}
\frac{\re\left(e^{-\ii\hth} (\omega - F(h))^n\right)}{\omega^n} &= \sum^{n-1}_{i=0} \lambda_i(\hth)\Lambda^i_{\omega} (F(h))^{i}
\end{align*}
for unique coefficients $\lambda_i(\hth)$. It follows that we may rewrite the equation in \eqref{coupled_dHYM} involving the scalar curvature as
\begin{equation*}
\Lambda_{\omega}\left(\operatorname{Ric}(\omega) - \alpha \lambda_1(\hth) F(h)\right) - \alpha \sum^{n-1}_{i=2} \lambda_i(\hth)\Lambda^{i}_{\omega} (F(h))^{i} = c  
\end{equation*} 
for a unique constant $c$. Now we can uniquely solve 
\begin{equation*}
\ii \Lambda_{\omega} \del\delbar h = \Delta_{\omega} h = \sum^{n-1}_{i = 2} \lambda_i(\hth)\Lambda^i_{\omega} F^{i}(h) - \sum^{n-1}_{i = 2} \lambda_i(\hth)\int_X \Lambda^i_{\omega} (F(h))^{i} \frac{\omega^n}{n!}
\end{equation*}
with the normalisation $\int_X h \omega^n = 0$. Thus we may rewrite our equation as
\begin{equation*}
\operatorname{Ric}(\omega) - \alpha \lambda_1(\hth) F(h) - \alpha \del\delbar h = \lambda\omega
\end{equation*}
for a unique $\lambda$, provided the cohomological condition
\begin{equation}\label{CohoMA}
 c_1(X)  = \lambda [\omega] + \alpha \lambda_1(\hth) [F(h)]
\end{equation}
is satisfied. 

In the special case of complex surfaces the computation above amounts to expressing $(F(h))^2$ in terms of $\omega^2$ and $F(h)\wedge \omega$ by using the dHYM equation. Moreover it is well known that the dHYM equation on surfaces reduces to a complex Monge-Amp\`ere equation (see e.g. the proof of Proposition \ref{aprioriPropSurfaces} below). Thus under the condition \eqref{CohoMA} on surfaces the coupled equations \eqref{coupled_dHYM} become the system of complex Monge-Amp\`ere equations
\begin{align}\label{RicciEquation2d}
\begin{cases}
\left(\ii\sin(\hth ) F(h) + \cos(\hth ) \omega\right)^2 = \omega^2\\
\operatorname{Ric}(\omega) = \lambda \omega + \frac{ \alpha}{\sin(\hth)} \ii F(h).
\end{cases}
\end{align}
With the assumption 
\begin{equation*}
\ii \sin(\hth ) F(h) + \cos(\hth ) \omega > 0
\end{equation*}
and provided the equalities
\begin{equation*}
\lambda = 1+\cos(\hth),\, \alpha = \sin^2(\hth)
\end{equation*}
hold, the system \eqref{RicciEquation2d} is given precisely by the coupled K\"ahler-Einstein equations in the sense of Hultgren-Witt Nystr\"om \cite{wytt}
\begin{align}\label{CoupledKE}
\begin{cases}
\operatorname{Ric}\left(\ii\sin(\hth ) F(h) + \cos(\hth ) \omega\right) = \operatorname{Ric}(\omega)\\
\operatorname{Ric}(\omega) =   \left(\ii\sin(\hth ) F(h) + \cos(\hth ) \omega\right) + \omega,
\end{cases}
\end{align}
of Fano type. 

Hultgren-Witt Nystr\"om (\cite{wytt} Theorem 1.7) showed that the K\"ahler-Einstein coupled equations are solvable on a K\"ahler-Einstein Fano manifold with discrete automorphisms, provided the corresponding decomposition of $c_1(X)>0$ is sufficiently close, in the K\"ahler cone, to a ``parallel" decomposition of the form $c_1(X) = \sum_i (\lambda_i c_1(X))$, $\lambda_i > 0$. By the discussion above, this implies immediately the following existence result for our coupled equations \eqref{coupled_dHYM}. It is convenient to set
$\chi = \ii\sin(\hth ) F(h) + \cos(\hth ) \omega$. 
\begin{cor}\label{CoupledKECor} Suppose $X$ is a del Pezzo surface with discrete automorphism group, and the conditions $[\chi] > 0$, $c_1(X) = [\omega] + [\chi]$ are satisfied. Assume moreover that the classes $[\omega]$, $[\chi]$ are sufficiently close, in the K\"ahler cone, to (positive) multiples of $c_1(X)$. Then there is a solution to our coupled equations \eqref{coupled_dHYM} in the classes $[F(h)]$, $[\omega]$, with coupling constant $\alpha = \sin^2(\hth)$.
\end{cor}
As the phase $e^{\ii \hth}$ depends on $[F(h)]$, $[\omega]$ (through explicit formulae which we give below in \eqref{SurfacePhase}), the conditions appearing in this Corollary are nonlinear constraints in these cohomology classes. To obtain examples in which they are satisfied we consider the choices
\begin{align*}
&[\omega] = \frac{1}{2} c_1(X) + t\eta,\\
&\ii[F(h)] = \frac{1}{2} c_1(X) - t\eta
\end{align*} 
where $\eta$ is a fixed class satisfying $\int c_1(X) \cup \eta = 0$ and the real parameter $t$ is sufficiently small. Then $[\omega]$, $\ii[F(h)]$ are positive and by \eqref{SurfacePhase} we have $\cos(\hth) = 0$, $\sin(\hth) = 1$. Since clearly $[\omega] + \ii[F(h)] = c_1(X)$, we do obtain solutions to the coupled equations \eqref{coupled_dHYM} in these classes, for all sufficiently small $t$.
\subsection{Duality} We have seen that in special cases our coupled equations on surfaces reduce to the coupled K\"ahler-Einstein equations \eqref{CoupledKE}, that is, setting 
\begin{equation*}
\chi = \ii\sin(\hth ) F(h) + \cos(\hth ) \omega,
\end{equation*} 
we obtain the equations
\begin{align*} 
\begin{cases}
\operatorname{Ric}(\chi) = \operatorname{Ric}(\omega)\\
\operatorname{Ric}(\omega) =  \chi + \omega.
\end{cases}
\end{align*}
We observe that these are now symmetric in $\chi$, $\omega$, so that the term $\chi$ involving the dHYM connection curvature $F(h)$ is interchanged with the variable K\"ahler form $\omega$. It could be interesting to interpret this duality in terms of the origin of the dHYM equation in the $B$-model. 
\subsection{Futaki invariant}
We find a first obstruction to the existence of solutions of the coupled equations \eqref{coupled_dHYM}, which generalises the classical Futaki character \cite{futaki}, following closely the approach of \cite{AGG}. Fix a complex line bundle $L \to X$ and the associated principal $GL(1,\mathbb{C})$-bundle $\mathcal{L}\to X$. 
Let $\mathscr{J}_\mathcal{L}$ be the space of holomorphic structures on $\mathcal{L}$, namely the integrable $GL(1,\mathbb{C})$-equivariant almost complex structures on $\mathcal{L}$, acting as multiplication by $\sqrt{-1}$ on the vertical bundle $V\mathcal{L} \cong \mathcal{L}\times \mathfrak{gl}(1,\mathbb{C}).$ An element $I \in \mathscr{J}_\mathcal{L}$ determines uniquely a complex structure $\check{I}$ on $X$ and a holomorphic structure on $L$. Let $\Aut(\mathcal{L},I)$ denote the group of automorphisms $g$ of the holomorphic principal bundle $(\mathcal{L},I)$, covering an automorphism $\check{g}$ of the complex manifold $(X, \check{I})$.  Each $\zeta \in \operatorname{Lie Aut}(\mathcal{L},I)$ covers a (unique) real holomorphic vector field $\check{\zeta}$ on $(X, \check{I})$. For any symplectic 2-form $\omega$ on $X$, which is $\check{I}$-compatible, we have the Hodge-type decomposition
\begin{equation*}
\check{\zeta} = \eta_{\phi_1} + \check{I}\eta_{\phi_2} + \beta,
\end{equation*}
where $\eta_{\phi_i}$ denotes the Hamiltonian vector field of $\phi_i\in C^{\infty}_0(X,\omega)$, while $\beta$ is the Riemannian dual of an harmonic 1-form, with respect to the metric $\omega( \cdot, \check{I}\cdot)$ (see \cite{lebrun}). Fixing also a Hermitian metric $h$ over the line bundle $L$, we define a $\mathbb{C}$-linear map
\begin{equation*}
\mathcal{F}_I : \operatorname{Lie}\Aut(\mathcal{L},I) \longrightarrow \mathbb{C}
\end{equation*}
given by
\begin{align*}
\mathcal{F}_I( \zeta ) = &\alpha \sqrt{-1} \int_X \theta_h \zeta\im\left(e^{-\ii\hth} (\omega - F(h))^n\right)\\
&- \int_X \phi \left( s(\omega)\omega^n -  \alpha \re\left(e^{-\ii\hth} (\omega - F(h))^n\right) \right) ,
\end{align*}
where $\phi = \phi_1 + \sqrt{-1}\phi_2$. One can show that $\mathcal{F}_I$ is a character of $\operatorname{Lie}\Aut(\mathcal{L},I)$ and does not depend on the choice of $\omega$ and $h$. The proof is essentially the same as in \cite{AGG}, Section 3, up to replacing $\omega^{\textrm{AB}}$ with $\omega^{\textrm{dHYM}}$. Then clearly $\mathcal{F}_I$ must vanish identically if the coupled equations \eqref{coupled_dHYM} have a solution. 
\subsection{Large radius limit} Let us consider the family of K\"ahler forms 
\begin{equation*}
\omega_t = t\omega,\, t\in \R_{>0}.
\end{equation*}
The \emph{large radius limit}, roughly mirror to a large complex structure limit on $\check{M}$, refers to the leading behaviour of the moment maps $\mu_{\widetilde{\G}}$, $\mu_{\alpha}$, computed with respect to $\omega_t$, as $t \to \infty$. The following result is proved in Section \ref{KYMSec}.
\begin{prop}\label{LargeRadiusProp} Let $F(h) = \ii F$, $c = \frac{n[\omega]^{n-1} \cup [F]}{[\omega]^n}$. As $t \to \infty$, there is an expansion
\begin{align*}
&\langle \mu_{\widetilde{\G}}(A), \zeta \rangle \\&= \frac{\ii}{n}\int_X \theta_A(\zeta)\left(\left(- n\omega^{n-1}\wedge F + c \omega^n\right)t^{n-1} + O(t^{n-3})\right) \\& +\frac{1}{n}\int_X p_*(\zeta) \\&\left(\omega^n t^n -\left(\frac{n(n-1)}{2} \omega^{n-2} \wedge F \wedge F - c n\omega^{n-1}\wedge F + \frac{1}{2}c^2\omega^n\right )t^{n-2}+ O(t^{n-4})\right) .    
\end{align*}
As a consequence, the K\"ahler-Yang-Mills coupled equations introduced in \cite{AGG}, in the particular case of line bundles, arise as the large radius limit of the coupled equations \eqref{coupled_dHYM}.
\end{prop}
Indeed, up to higher order terms as $t \to \infty$, the system \eqref{coupled_dHYM}, becomes in the large radius limit
\begin{align*}
\begin{cases}
 n\omega^{n-1}\wedge F = c\omega^n\\
 s(\omega_t) - \alpha \frac{\omega^n t^n -\left(\frac{n(n-1)}{2} \omega^{n-2} \wedge F \wedge F - c n\omega^{n-1}\wedge F + \frac{1}{2}c^2\omega^n\right )t^{n-2}}{(\omega_t)^n} = f,
 \end{cases}
\end{align*} 
or equivalently
\begin{align*}
\begin{cases}
 \Lambda_{\omega} F = c\\
 t^{-1}s(\omega) - \alpha \left(1 - t^{-2} \Lambda^2_{\omega} (F\wedge F) -\frac{1}{2}t^{-2}c^2 \right)  = f.
 \end{cases}
\end{align*} 
Thus choosing the appropriate scaling behaviour for the coupling constant and datum,
\begin{equation*}
\alpha = \alpha' t,\quad f = -\alpha' t + \frac{\tilde{f}}{t}
\end{equation*}
we arrive at the equations
\begin{align}\label{KYMIntro}
\begin{cases}
 \Lambda_{\omega} F = c\\
 s(\omega) +\alpha' \left(\Lambda^2_{\omega} (F\wedge F) + \frac{1}{2} c^2 \right)  = \tilde{f}.
\end{cases}
\end{align} 
When $\tilde{f}$ is a (topologically fixed) constant, these are precisely the \emph{coupled K\"ahler-Yang-Mills equations} studied in \cite{AGG}, in the particular case of a holomorphic line bundle. 
\subsection{Small radius limit} The \emph{small radius limit} refers to the leading behaviour of the moment maps $\mu_{\widetilde{\G}}$, $\mu_{\alpha}$, computed with respect to $\omega_t$, as $t \to 0$.
\begin{prop}\label{SmallRadiusProp} Let $F(h) = \ii F$, $c = \frac{n [\omega]\cup [F]^{n-1}}{[F]^n}$. As $t \to 0$, there is an expansion
\begin{align*}
\langle \mu_{\widetilde{\G}}(A), \zeta \rangle &= \frac{\ii}{n}\int_X \theta_A(\zeta)\left(\left(n\omega\wedge F^{n-1}-c F^n\right)t + O(t^{2})\right) \\& +\frac{1}{n}\int_X p_*(\zeta) \left( F^n + O(t^{3})\right) .    
\end{align*}
\end{prop}
Proposition \ref{SmallRadiusProp} is proved in Section \ref{KYMSec}. Thus, up to higher order terms as $t \to 0$, the system \eqref{coupled_dHYM}, becomes in the small radius limit
\begin{align*}
\begin{cases}
 n\omega \wedge F^{n-1} = c F^n\\
 s(\omega_t) - \alpha \frac{F^n}{(\omega_t)^n} = f,
 \end{cases}
\end{align*} 
or equivalently, provided $F$ is a K\"ahler form,
\begin{align*}
\begin{cases}
 \Lambda_{F} \omega = c\\
 t^{-1}s(\omega) - \alpha \frac{F^n}{\omega^n t^n}  = f.
\end{cases}
\end{align*} 
With the appropriate scaling behaviour
\begin{equation*}
\alpha = t^{n-1}\alpha',\quad f = t^{-1}\tilde{f}
\end{equation*}
we arrive at the system
\begin{align}\label{SmallRadiusEqu}
\begin{cases}
\Lambda_{F} \omega = c\\
s(\omega) - \alpha' \frac{F^n}{\omega^n}  = \tilde{f}.
\end{cases}
\end{align}
This comprises the \emph{J-equation} of Donaldson \cite{donaldsonJ} and Chen \cite{chen}, well-known to be a scaling limit of the dHYM equation (see e.g. \cite{collinsXie}). However, unlike the large radius limit, it seems that the system \eqref{SmallRadiusEqu} does not appear in the literature, except for the case when $X$ is a complex surface. In that case setting
\begin{equation*}
\chi = c F - \omega
\end{equation*}
we may rewrite \eqref{SmallRadiusEqu} as the system
\begin{align*} 
\begin{cases}
 \chi^2 = \omega^2\\
 s(\omega) - \frac{\alpha'}{c^2}\Lambda_{\omega}\chi = \tilde{f} + \frac{\alpha'}{c^2}.
 \end{cases}
\end{align*}
In the particular case when $\tilde{f}$ is a constant this comprises a complex Monge-Amp\`ere equation coupled to a twisted cscK equation, and it is precisely of the type studied by Datar and Pingali \cite{datar}.
 
\subsection{The equations on abelian varieties} After establishing the general results described above, in the rest of the paper we focus on studying the coupled equations \eqref{coupled_dHYM}, and their scaling limits, when $X$ is an abelian variety. Note that in this special case the equations \eqref{coupled_dHYM} for constant $f$ are always solvable by taking constant coefficients representatives, so in fact it is necessary here to include a suitable non-constant source term. 

Considering abelian varieties is motivated in part by the origin of the dHYM equations in the $B$-model: for example, homological mirror symmetry for abelian varieties has been studied in detail \cite{fukaya, polish}; moreover abelian varieties also play a special role in this context as fibres of holomorphic Lagrangian fibrations (see e.g. \cite{gross}). 

A more analytic reason is that the coupled equations \eqref{coupled_dHYM} interact nicely with the theory of the scalar curvature of K\"ahler metrics on complex tori, or more generally of periodic solutions of Abreu's equation, as developed by Feng and Sz\'ekelyhidi \cite{fengSz} (see also \cite{lejmiSz}). This is exploited in our results Theorems \ref{MainThmSurface}, \ref{MainThmLargeRadius} and \ref{MainThmODE}.   

Further motivation comes from the fact that the problem of realising solutions of the coupled equations \eqref{coupled_dHYM} effectively in the $B$-model, as in Remark \ref{realizRmk}, is more tractable when $X$ is a complex torus. We explain this, in a special case, in Proposition \ref{HKProp}. 

We can now discuss our existence results on abelian varieties. As in the work of Feng and Sz\'ekelyhidi \cite{fengSz} we may assume, without loss of generality, that $X$ is the abelian variety $\C^n/(\Z^n\oplus \sqrt{-1}\Z^n)$ and $[\omega_0]$ is the class of the constant metric $\omega_0 =  \sqrt{-1} \sum_i dz_i \wedge \overline{dz}_i$. The group $U(1)^n$ acts on $X$, by translations in the direction $\sqrt{-1}\R^n$. We will restrict to $U(1)^n$-invariant tensors and thus work effectively over the real torus $T = \R^n/\Z^n$. Note that an invariant K\"ahler metric $g$ is given by the real Hessian of a convex function $v(y)$ on $\R^n$ of the form
\begin{equation*}
v(y) = \frac{1}{2}|y|^2 + \psi(y)
\end{equation*}
where $\psi\!:\R^n \to \R$ is $\Z^n$-periodic (so we have $\omega_g = \sqrt{-1} \sum_{i,j}  v_{ij}dz_i\wedge \overline{dz}_j$). Such a function has a well-defined \emph{Legendre transform} $u(x)\!:\R^n \to \R$, where the ``symplectic coordinates" $x$ and ``real holomorphic" coordinates $y$ are related by the diffeomorphism $y = \nabla u(x)$.

We begin by studying the special case when $X$ is an \emph{abelian surface}. As a preliminary step we derive a priori estimates for invariant solutions of \eqref{coupled_dHYM}, under a semipositivity condition, and a condition on the phase $e^{\ii\hth}$. These rely heavily on the results of \cite{fengSz} and may be of independent interest.

As before, it is convenient to set $F(h) = \ii F$, for a real $2$-form $F$, and to formulate our results in terms of $F$.  
\begin{prop}\label{aprioriPropSurfaces} Let $X$ be an abelian surface, and $(\omega_g, F)$ be a $U(1)^2$-invariant solution of the coupled equations \eqref{coupled_dHYM}, for a fixed function $f$. Suppose we have $F \geq 0$, and the phase $e^{\ii\hth}$ satisfies
\begin{equation*}
\sin(\hth) < 0,\, \cos(\hth) > 0.
\end{equation*} 
Then there are a priori $C^k$ estimates on $(g, F)$ of all orders, with respect to the background flat metric $\omega_0$, depending only on $f$, the phase $e^{\ii\hth}$, and the coupling constant $\alpha$. Moreover the metric $g$ is uniformly positive, depending on $\sup |f|$, $\alpha$ and $e^{\ii \hth}$.
\end{prop}
\begin{rmk} The necessary conditions that the class $[F]$ is (semi)positive and that we have $\sin(\hth) < 0,\, \cos(\hth) > 0$ are indeed compatible. A straightforward computation shows that on a complex surface we have
\begin{align}\label{SurfacePhase}
&\nonumber \cos(\hth) = \frac{\int \omega^2 - F^2}{\left(\left(\int \omega^2 - F^2\right)^2 + 4\left(\int F \wedge \omega \right)^2\right)^{1/2}},\\
&\sin(\hth) = -\frac{2  \int F \wedge \omega}{\left(\left(\int \omega^2 - F^2\right)^2 + 4\left(\int F \wedge \omega \right)^2\right)^{1/2}},
\end{align}
so it is enough to choose a positive class $[F]$ with smaller volume than $[\omega]$.
\end{rmk}

Proposition \ref{aprioriPropSurfaces} is proved is Section \ref{surfaceSec}. Similarly to the work of Feng-Sz\'ekelyhidi \cite{fengSz} in the case of the scalar curvature, using the Legendre transform we can apply Proposition \ref{aprioriPropSurfaces} to obtain an existence result for the coupled equations \eqref{coupled_dHYM} on an abelian surface $X$. Let $A$ be any $U(1)^2$-invariant function on $X$, satisfying the necessary cohomological condition
\begin{equation*}
\int_{X} A \frac{\omega^2}{2} = - \alpha \int_X \re\left(e^{-\ii\hth} (\omega - F(h))^2\right).
\end{equation*}
\begin{thm}\label{MainThmSurface} Let $X$ be an abelian surface with K\"ahler classes $[\omega]$, $[F]$. Suppose the phase $e^{\ii\hth}$ satisfies
\begin{equation*}
\sin(\hth) < 0,\, \cos(\hth) > 0.
\end{equation*}
Consider the equations \eqref{coupled_dHYM}, with coupling constant 
\begin{equation*}
\alpha = \alpha' \cos(\hth),\,\alpha'>0 
\end{equation*}
and datum $f$ given by the image of any function $A$, as above, under Legendre duality, that is
\begin{equation*}
f(\nabla u(x)) = A(x).
\end{equation*}  
Then, these are solvable provided the classes $[F]$ and $[\omega]$ are sufficiently close, depending only on $\alpha'$ and $\sup|A|$.
\end{thm}
Theorem \ref{MainThmSurface} is proved is Section \ref{surfaceSec}. The following application follows at once, by rescaling suitably (recall Remark \ref{scalingRmk}). 
\begin{cor} Fix negative line bundles $L, N$ on the abelian surface $X$. Then for all sufficiently large $k$, depending only on $\alpha'$, $A$, the equations \eqref{coupled_dHYM} with coupling constant $\alpha' \cos(\hth) $ and datum $f$ as in Theorem \ref{MainThmSurface} are solvable on the line bundle $L^{\otimes k}\otimes N^{-1}$, with respect to the K\"ahler class $-k c_1(L)$. 
\end{cor}
Theorem \ref{MainThmSurface} suggests a similar approach, based on the positivity of $L^{-1}$, in the case of the large radius limit, that is, of the K\"ahler-Yang-Mills equations studied in \cite{AGG}, 
\begin{align}\label{prescribedKYM}
\begin{cases}
\Lambda_g F = \mu\\
s(g) + \alpha \Lambda^2_g ( F  \wedge F ) = f,  
\end{cases}
\end{align}
for a smooth function $f$. Indeed it turns out that in this case we can obtain analogues of Proposition \ref{aprioriPropSurfaces} and Theorem \ref{MainThmSurface}, for arbitrary dimension. Following \cite{fengSz}, as above, we may assume, without loss of generality, that $X$ is the abelian variety $\C^n/(\Z^n\oplus \sqrt{-1}\Z^n)$ and $[\omega_0]$ is the class of the metric $\omega_0 =  \sqrt{-1} \sum_i dz_i \wedge \overline{dz}_i$. 
\begin{prop}\label{aprioriProp} Let $(g, h)$ be a $U(1)^n$-invariant solution of \eqref{prescribedKYM} on a line bundle $L\to M$, for a fixed function $f$. Suppose we have $F\geq 0$. Then there are a priori $C^k$ estimates on $(g, h)$ of all orders, with respect to the background flat metric $\omega_0$, depending only on $f$, the dimension $n$, the degree $\mu$ of $L^{-1}$, and the coupling constant $\alpha$. Moreover the metric $g$ is uniformly positive, depending only on $\sup |f|$, $\alpha$ and $\mu$.
\end{prop}
 Let $A$ be any $U(1)^n$-invariant function on $X$, satisfying 
\begin{equation*}
\int_{M} A \frac{\omega^n}{n!} = \lambda \frac{[\omega]^n}{n!},
\end{equation*}
where the topological constant $\lambda$ is given by
\begin{equation*}
\lambda = \alpha n (n-1) \frac{(c_1(L))^2 \cup [\omega]^{n-2}}{[\omega]^n}.
\end{equation*} 
Suppose $[\omega]$ is the class of the curvature of an ample line bundle $N$ on $M$. We also fix a negative line bundle $L$, with degree $-\mu$.
\begin{thm}\label{MainThmLargeRadius} There exists an integer $K > 0$ such that, for all $k \geq K$, there are a Hermitian metric $h$ on the fibres of $L^n \otimes (k-1) \mu N$ and an invariant K\"ahler metric $g \in c_1(N)$, solving the equations 
\begin{align*}
\begin{cases}
 \Lambda_g F = n k \mu\\
 s(g) + \frac{\alpha}{(n k)^2} \Lambda^2_g ( F \wedge F ) = f,
 \end{cases}
\end{align*}   
where $f$ is the image of any function $A$, as above, under Legendre duality, that is
\begin{equation*}
f(\nabla u(x)) = A(x).
\end{equation*}
The integer $K$ depends only on $\sup |f|$, the dimension $n$, the degree $\mu$, and the coupling constant $\alpha$.
\end{thm}
Note that Theorem \ref{MainThmLargeRadius} is not obtained by ``perturbation" around $\alpha = 0$ and the solution of the corresponding problem $s(g) = f$ found by Feng-Sz\'ekelyhidi. Indeed for all $k$ as in the statement we have  
\begin{equation*}
\int_M \Lambda^2_g ( F  \wedge F )\frac{\omega^n_g}{n!} = (n k)^2 \lambda = O(k^2),
\end{equation*}
so the term $\frac{\alpha}{(n k)^2} \Lambda^2_g ( F \wedge F )$ coupling the metric to the connection is of order $O(1)$ and the actual coupling constant is $\alpha$, not $\frac{\alpha}{k^2}$.

Theorems \ref{MainThmSurface} and \ref{MainThmLargeRadius} apply when the datum $f$ (or rather its Legendre transform $A$ with respect to the unknown $g$) is arbitrary. We can obtain much more precise results when $A$ has a particular form. We consider here the case when $A$ is invariant under translations with respect to all but one of the symplectic coordinates, say $x_1$. For simplicity we analyse the case when $X$ is the abelian surface $\C^2/(\Z^2\oplus\ii\Z^2)$, although similar results hold much more generally.
\begin{thm}\label{MainThmODE} Suppose $X$ is the abelian surface $\C^2/(\Z^2\oplus\ii\Z^2)$. Then, the coupled equations \eqref{coupled_dHYM} are solvable on any $L$, with respect to the class $[\omega_0]$, with coupling constant $\alpha >0$, and datum $f(y_1)$ given by the image of any function $A(x_1)$, as above, under Legendre duality, that is
\begin{equation*}
f(\nabla u(x)) = A(x_1).
\end{equation*}
\end{thm}
We also prove an analogue of this result for the large and small radius limits. 
\begin{thm}\label{LargeSmallRadiusThmODE} In the situation of Theorem \ref{MainThmODE}, the coupled K\"ahler-Yang-Mills equations \eqref{KYMIntro} are always solvable on $L$, with respect to the class $[\omega_0]$, with coupling constant $\alpha >0$, and datum $f(y_1)$ given by the image of any function $A(x_1)$ under Legendre duality. 

The same holds for the small radius limit coupled equations \eqref{SmallRadiusEqu}, under the condition $\det(F^0)  > 0$, where $F^0(h) =  \sum_{i,j} F^0_{ij} dz_i\wedge\overline{dz}_j$ is a \emph{constant} curvature form for a metric on $L$.
\end{thm}
Let us return to the question of realising the solutions of the coupled equations \eqref{coupled_dHYM} effectively in the $B$-model. As recalled in Remark \ref{realizRmk}, it is always possible to embed $(X, \omega)$ isometrically as the zero section of the holomorphic cotangent bundle $T^*X$, endowed with a hyperk\"ahler metric defined in a \emph{formal} neighbourhood of the zero section. It is natural to ask when this metric extends at least to an open neighbourhood of the zero section, in the analytic topology. By the main result of \cite{feix}, this is the case if and only if $\omega$ is real analytic.
\begin{prop}\label{HKProp} Suppose the datum $A(x_1)$ is real analytic. Then, the metric $\omega$ underlying a solution of the coupled equations \eqref{coupled_dHYM} given by Theorem \ref{MainThmODE} is also real analytic. It follows that these solutions can be realised effectively in the $B$-model of an open neighbourhood of $X \subset T^*X$ endowed with a hyperk\"ahler metric, extending $\omega$.
\end{prop}  
The proof is given in Section \ref{ODESec}. Naturally it would be interesting to understand when these extensions are complete.\\

The Appendix to this paper is devoted to the linearised K\"ahler-Yang-Mills equations in symplectic coordinates on a torus. In particular we prove that these linearised equations correspond to a scalar linear differential operator which has trivial kernel and is formally self-adjoint, with respect to the Lebesgue measure. Besides its application in our proof of Theorem \ref{MainThmLargeRadius}, we believe this may be a useful result in view of future applications.\\

\section{Extended gauge group and scalar curvature}\label{gaugeSec}
This Section is devoted to the proofs of Theorem 
\ref{ExtendedGaugeThm} and Corollary \ref{ExtendedGaugeCor}. Let $X$ be a compact $n$-dimensional K\"ahler manifold, with K\"ahler form $\omega$, and $L \to X$ a complex line bundle with a Hermitian metric $h$. We consider the space $\mathcal{A}$ of $h$-unitary connections on $L$, endowed with the symplectic structure given by $\omega^{\text{dHYM}}$.
The $\widetilde{\mathcal{G}}$-equivariant map $\theta$ defined as
\begin{align*}
\theta\colon &\mathcal{A} \to \text{Hom(Lie }\widetilde{\mathcal{G}},\text{Lie }\mathcal{G})\\
& A \mapsto \theta_A
\end{align*}
associates to each connection $A$ the projection operator $\theta_A$ introduced in \eqref{VertProj}.
We consider also the map $\theta^{\perp}$ given by
\begin{align*}
\theta^{\perp}\colon &\mathcal{A} \to \text{Hom(Lie }\mathcal{H},\text{Lie }\widetilde{\mathcal{G}})\\
& A \mapsto \theta^{\perp}_A
\end{align*}
where the \textit{lifting operator} $\theta^{\perp}_A$ is uniquely defined by $\text{Id} = \iota \circ \theta_A + \theta^{\perp}_A\circ p$, with $\iota$ and $p$ as in \eqref{LieSplitting}.
For any $\zeta \in \text{Lie }\widetilde{\mathcal{G}}$, $Y_{\zeta}$ denotes the vector field on $\mathcal{A}$ associated to the infinitesimal $\widetilde{\mathcal{G}}$-action on $\mathcal{A}\colon$
\begin{equation*}
Y_{\zeta | A} = \frac{d}{d\text{t}}\Big|_{ t=0}\text{exp}({t\zeta})\cdot A.
\end{equation*}
In particular we have
\begin{equation*}
Y_{\theta^{\perp}\eta_{\phi} | A} = -\eta_{\phi} \mathrel{\lrcorner} F_A
\end{equation*}
for any $\eta_{\phi} \in \text{Lie }\mathcal{H}$ (see e. g. \cite{AGG}, Lemma 1.5).
It follows from \cite{AGG}, Proposition 1.3, that the $\widetilde{\mathcal{G}}$-action on $\mathcal{A}$ is Hamiltonian if and only if there is a $\widetilde{\mathcal{G}}$-equivariant map $\sigma \colon \mathcal{A} \to (\text{Lie }\mathcal{H})^*$ satisfying
\begin{equation}\label{ModifiedMomentCondition}
\omega^{\text{dHYM}}(Y_{\theta^{\perp}\eta_{\phi} | A}, a) = \langle \mu_{\mathcal{G}}, a(\eta_{\phi}) \rangle + d\langle \sigma, \eta_{\phi} \rangle (a).
\end{equation}
We claim that this holds for the equivariant map defined by
\begin{equation*}
\langle \sigma(A), \eta_{\phi}\rangle = \frac{1}{n} \int_X \phi \re  \left(e^{-\ii\hth} (\omega - F(A))^n\right),
\end{equation*}
for all $\eta_{\phi}\in \text{Lie }\mathcal{H}$.

Since $\sqrt{-1}a\wedge \im  \left(e^{-\ii\hth} (\omega - F(A))^n\right)=0$, contracting with $\eta_{\phi}$ we obtain the identity 
\begin{align}\label{SigmaIdentity}
\nonumber &-\sqrt{-1}n a \wedge \left(\eta_{\phi}\mathrel{\lrcorner} \omega \right)\wedge \im  \left(e^{-\ii\hth} (\omega - F(A))^{n-1}\right)\\
&\nonumber+n a \wedge \left(\eta_{\phi}\mathrel{\lrcorner} F(A) \right)\wedge \re  \left(e^{-\ii\hth} (\omega - F(A))^{n-1}\right)\\
& +\sqrt{-1}a(\eta_{\phi})\im  \left(e^{-\ii\hth} (\omega - F(A))^n\right)=0.
\end{align}
On the other hand, by the above identity for the infinitesimal generator, we have
\begin{equation*}
\omega^{\text{dHYM}}(Y_{\theta^{\perp}\eta_{\phi} | A}, a) = -\int_X a \wedge \left(\eta_{\phi}\mathrel{\lrcorner} F(A) \right)\wedge \re  \left(e^{-\ii\hth} (\omega - F(A))^{n-1}\right).
\end{equation*}
Also, by definition,
\begin{equation*}
\langle \mu_{\mathcal{G}}, a(\eta_{\phi}) \rangle = +\frac{\sqrt{-1}}{n} \int_X a(\eta_{\phi})\im  \left(e^{-\ii\hth} (\omega - F(A))^n\right),
\end{equation*}
and similarly
\begin{align*}
d\langle \sigma, \eta_{\phi} \rangle (a) &=- \frac{1}{n} \int_X \phi \re \left(n\, da \wedge e^{-\ii\hth} (\omega - F(A))^{n-1}  \right) \\
&=-\sqrt{-1}\int_X \phi da \wedge \im \left(e^{-\ii\hth}(\omega - F(A))^{n-1}  \right)\\
&=-\sqrt{-1}\int_X a \wedge \left(\eta_{\phi}\mathrel{\lrcorner} \omega \right) \wedge \im  \left(e^{-\ii\hth} (\omega - F(A))^{n-1}\right).
\end{align*}
Hence, using the identity \eqref{SigmaIdentity}, we see that the condition \eqref{ModifiedMomentCondition} is satisfied and Theorem \ref{ExtendedGaugeThm} follows.

Now we endow $\mathscr{A}\times \mathscr{J}$ with the infinite dimensional K\"ahler structure given by the form
\begin{equation*}
\omega_{\alpha} = \alpha\omega^{\text{dHYM}} + \omega^{\text{DF}}
\end{equation*}
with $\alpha > 0$. The $\widetilde{\mathcal{G}}$-action on $\mathscr{A}\times \mathscr{J}$ preserves $\mathscr{P}$, and combining our computation above with the well-known results of Donaldson \cite{donaldson} and Fujiki \cite{fujiki} we see that it is Hamiltonian, with equivariant moment map $\mu_{\alpha}\!: \mathscr{A}\times \mathscr{J} \to (\operatorname{Lie} \widetilde{\mathcal{G}} )^*$ given by
\begin{align*}
\langle \mu_{\alpha}(A,J), \zeta \rangle &=- \int_X p_*(\zeta) s(J) \omega^n + \alpha \langle \mu_{\widetilde{\G}}(A), \zeta \rangle\\     &=\sqrt{-1} \alpha \int_X \theta_A \zeta \im\left(e^{-\ii\hth} (\omega - F(A))^n\right) \\
&-\int_X \phi \left( s(J)\omega^n - \alpha \re\left(e^{-\ii\hth} (\omega - F(A))^n\right) \right),
\end{align*}
for all $(A,J)\in \mathscr{A}\times \mathscr{J}$, $\zeta \in \operatorname{Lie} \widetilde{\mathcal{G}}$ and $p_*(\zeta) = \eta_{\phi}$.   

\section{A priori estimates and Theorem \ref{MainThmSurface}}\label{surfaceSec}
We consider the coupled equations \eqref{coupled_dHYM} when $X$ is an abelian surface $\C^2/\Lambda$. In particular we will soon assume that $L$ is semipositive. In order to simplify our exposition, following \cite{fengSz}, we can further assume that $X = \C^2/(\Z^2 \oplus \sqrt{-1}\Z^2)$, and that the background K\"ahler form is $\omega_0 =  \sqrt{-1} \sum_i dz_i \wedge \overline{dz}_i$. The general case only differs by slightly more complicated notation.

The group $U(1)^2$ acts on $X$, by translations in the direction $\sqrt{-1}\R^n$, and we will look for invariant solutions, so we are effectively considering equations on the real torus $\R^2/\Z^2$. Following \cite{fengSz}, Section 5 we may formulate the problem (using invariant \emph{complex coordinates}, with real part $y$) in terms of a convex function
\begin{equation*}
v(y) = \frac{1}{2}|y|^2 + \psi(y)
\end{equation*}
where $\psi\!:\R^n \to \R$ is periodic, with fundamental domain $\Omega = [0,1]\times[0,1]$, normalised by $\psi(0)=0$. The invariant metric $g$ is given by the real Hessian of $v(y)$, namely $\omega_g = \sqrt{-1} \sum_{i,j}  v_{ij}dz_i\wedge \overline{dz}_j$. Then, by a standard formula in K\"ahler geometry, we have
\begin{align*}
s(g) = -\frac{1}{4} v^{ij}[\log\det(v_{ab})]_{ij}.
\end{align*} 
We set $F(h) = \ii F$, and abusing notation slightly we think of $F$ as a periodic function with values in symmetric matrices (so we have $F(h) = \sum_{i,j} F_{ij} dz_{i}\wedge \overline{dz}_j$). Then the coupled equations \eqref{coupled_dHYM} become 
\begin{align}\label{coupled_dHYM_torus}
\begin{cases}
\im\left(e^{-\ii\hth}\det(v_{ij} - \ii F_{ij})\right) = 0\\
 -\frac{1}{4} v^{ij}[\log\det(v_{ab})]_{ij} -\alpha \frac{\re\left(e^{-\ii\hth}\det(v_{ij} - \ii F_{ij})\right)}{\det(v_{ab})} = f,  
 \end{cases}
\end{align}
where the phase $e^{\ii \hth}$ is determined by the cohomological condition
\begin{equation}\label{Phase}
\int_{\Omega} \det(v_{ij} - \ii F_{ij}) \in \R_{>0}e^{\ii\hth} 
\end{equation}
and the datum $f$ must satisfy the necessary cohomological constraint 
\begin{equation}\label{constr}
\int_{\Omega} f \det(v_{ab}) = -\alpha \int_{\Omega}\re\left(e^{-\ii\hth}\det(v_{ij} - \ii F_{ij})\right).   
\end{equation}
It is especially convenient to formulate the problem in terms of \emph{symplectic coordinates}, that is, in terms of the Legendre transform $u(x)$ of $v(y)$, see \cite{fengSz} Section 5. Recall that $u(x)$ is defined by the equation
\begin{equation*}
u(x) + v(y) = x\cdot y,
\end{equation*}
where we set $x = \nabla v(y)$. Then $u(x)$ has the form
\begin{equation}\label{potential}
u(x) = \frac{1}{2}|x|^2 + \phi(x).
\end{equation}   
where $\phi\!:\R^n \to \R$ is periodic, with domain $\Omega$, and we have the inverse relation
\begin{equation*}
y = \nabla u(x).
\end{equation*} 
Using a well-known result of Abreu for the scalar curvature in symplectic coordinates (see \cite{abreu}), as well as the fundamental Legendre duality property 
\begin{equation*}
v^{ij}(y) = u_{ij}(x),
\end{equation*} 
the coupled equations \eqref{coupled_dHYM} become
\begin{align}\label{coupled_dHYM_symplectic}
\begin{cases}
 \im\left(e^{-\ii\hth}\det(u^{ij} - \ii F_{ij}(\nabla u))\right) = 0\\
 -\frac{1}{4} [u^{ij}]_{ij} -\alpha  \re\left(e^{-\ii\hth}\det(u^{ij} - \ii F_{ij}(\nabla u))\right) \det(u_{ab})  = A,
 \end{cases}  
\end{align} 
where the datum $A$ is given by the relation $A(x) = f(y) = f(\nabla u)$. 

A key advantage of formulating the problem in symplectic coordinates is that it is now trivial to take the cohomological constraint \eqref{constr} into account: the datum $A(x)$ must simply satisfy
\begin{equation*}
\int_{\Omega} A(x) d\mu(x) = -\alpha \int_{\Omega}\re\left(e^{-\ii\hth}\det(v_{ij} - \ii F_{ij})\right)
\end{equation*}
with respect to the \emph{fixed} Lebesgue measure $d\mu(x)$.

In the following we start from the equations is symplectic coordinates with datum $A(x)$ and \emph{define} $f(y)$ through the relation 
\begin{equation*}
f(\nabla u(x)) = A(x).
\end{equation*}
Our first task is to establish the necessary a priori estimates for this problem, Proposition \ref{aprioriProp}. As in \cite{fengSz}, the first step for this is obtaining a uniform bound for the determinant.
\begin{lem}\label{detEstimate} Suppose the phase $e^{\ii\hth}$ satisfies
\begin{equation*}
\sin(\hth) < 0,\, \cos(\hth) > 0.
\end{equation*}
Then, there are uniform constants $c_1, c_2 > 0$, depending only on the coupling constant $\alpha$, the phase $e^{\ii\hth}$ and on $\sup |f|$, such that a solution of \eqref{coupled_dHYM_torus} with $F  \geq 0$ satisfies  
\begin{equation*}
0 < c_1 < \det(v_{ab}) < c_2.
\end{equation*}
(Recall $A(x), f(y)$ and $u(x), v(y)$ are related by Legendre transform). 
\end{lem}
\begin{proof} Feng-Sz\'ekelyhidi \cite{fengSz} study Abreu's equation
\begin{equation*}
[u^{ij}]_{ij} = \tilde{A}
\end{equation*} 
in all dimensions and for an arbitrary smooth periodic function $\tilde{A}$ (with zero average). In loc. cit. Section 3 it is shown that solutions satisfy a uniform bound of the form
\begin{equation*}
0 < c'_1 < \det(u_{ab}) < c'_2
\end{equation*}
where the constants $c'_1$, $c'_2$ depend only on the dimension $n$ and a bound on $\sup |\tilde{A}|$. In our case, we can write the equations in symplectic coordinates \eqref{coupled_dHYM_symplectic} in the form $[u^{ij}]_{ij} = \tilde{A}$ with the choice
\begin{equation*}
\tilde{A} = -4 \alpha  \re\left(e^{-\ii\hth}\det(u^{ij} - \ii F_{ij}(\nabla u))\right) \det(u_{ab}) -4 A.
\end{equation*}
We claim that, under our assumptions, there is a uniform a priori bound for $\sup|\tilde{A}|$, depending only on $\sup |A|$, $\alpha$, $e^{\ii\hth}$. Equivalently, we claim that there is a uniform bound for the quantity
\begin{equation*}
\frac{\re\left(e^{-\ii\hth}\det(v_{ij} - \ii F_{ij})\right)}{\det(v_{ab})} = \frac{\re(e^{-\ii \hth}(\omega - \ii F)^2)}{\omega^2}.
\end{equation*}
In order to see this, note that in the two-dimensional case the coupled equations \eqref{coupled_dHYM} may be written as
\begin{align*}
 \begin{cases}
  F^2 \sin (\hth) - 2 F\wedge \omega  \cos (\hth)-\omega^2 \sin(\hth) =0\\
  s(\omega) - \alpha \frac{ -F^2 \cos(\hth) - 2 F\wedge\omega  \sin(\hth) +\omega ^2 \cos (\hth)}{\omega^2}=f. 
  \end{cases}
\end{align*}
In particular the dHYM equation is 
\begin{equation*}
 \sin (\hth )\frac{F^2}{\omega^2} - 2 \frac{F\wedge \omega}{\omega^2}  \cos (\hth ) = \sin (\hth ).
\end{equation*}
So under the conditions $\sin(\hth) < 0$, $\cos(\hth) > 0$, together with semipositivity $F \geq 0$, the dHYM equation implies the a priori bounds
\begin{equation}\label{aprioriPhase}
\frac{F^2}{\omega^2} < 1, \frac{F\wedge \omega}{\omega^2} < \frac{|\tan( \hth)|}{2}, 
\end{equation}
which immediately give the required bounds on $\tilde{A}$.  
\end{proof}
It is possible to obtain higher order estimates from the bound on the determinant given by Lemma \ref{detEstimate}. Following \cite{fengSz} Section 4 the key idea is to write the second equation in \eqref{coupled_dHYM_symplectic} in the form
\begin{equation}\label{linearMA}
U^{ij} w_{ij} = \tilde{A} 
\end{equation}
where $U^{ij}$ is the cofactor matrix of the Hessian $u_{ij}$, while 
\begin{align*}
& w = (\det(u_{ab}))^{-1},\\
& \tilde{A} = -4 \alpha  \re\left(e^{-\ii\hth}\det(u^{ij} - \ii F_{ij}(\nabla u))\right) \det(u_{ab}) -4 A.
\end{align*}
Note that this rewriting is possible because of the identity $[U^{ij}]_i = 0$. Then \eqref{linearMA} can be regarded as a non-homogeneous linearised Monge-Amp\`ere equation satisfied by $w$. 
\begin{lem}\label{highEstimatesLem} Suppose the phase $e^{\ii\hth}$ satisfies
\begin{equation*}
\sin(\hth) < 0,\, \cos(\hth) > 0.
\end{equation*}
Then, there are uniform constants $0 < \Lambda_0 < \Lambda_1$, $0<\delta < 1$, $\Lambda_{2, \delta} >0$, depending only on $e^{\ii\hth}$, $\alpha$ and $A$, such that for a solution $u$, $F$ of  \eqref{coupled_dHYM_symplectic} with $F\geq 0$ we have
\begin{equation*}
\Lambda_0 I < u_{ij} < \Lambda_1 I,\, ||u||_{C^{2, \delta}} < \Lambda_{2,\delta}. 
\end{equation*}
\end{lem}
\begin{proof} We first observe that by \cite{fengSz} Lemma 4 we have a uniform $C^1$ bound on $u$. This is independent of the equation satisfied by $u$ and holds simply by periodicity, positivity $u_{ij} > 0$ and the normalisation $\psi(0) = 0$. In particular we have a uniform $C^0$ bound on $u$.

Let us now show that we have an estimate on $|| u ||_{C^{2,\delta}}$ for some $\delta > 0$, depending only on $\sup|A|$, $\alpha$, $e^{\ii\hth}$. In \cite{TrudWang} Section 3.7, Corollary 3.2, Trudinger-Wang give an interior H\"older estimate for a $C^2$ solution of the non-homogeneous linearised Monge-Amp\`ere equation on the ball $B_1(0) \subset \R^n$, for all $n$. Note that, adapting to our present notation, they would actually write \eqref{linearMA} as
\begin{equation*} 
U^{ij} w_{ij} = \left(\frac{\tilde{A}}{\det(u_{ab})}\right)\det(u_{ab}). 
\end{equation*}
Then their estimate takes the form
\begin{equation*}
|| w ||_{C^{\delta}(B_{1/2}(0))} \leq C \left(|| w ||_{C^0(B_{1}(0))} + \int_{B_{1}(0)} \Big|\frac{\tilde{A}}{\det(u_{ab})}\Big|^n \det(u_{ab}) d\mu \right),
\end{equation*}
where $d\mu$ is the Lebesgue measure (so $\det(u_{ab}) d\mu$ is the Monge-Amp\`ere measure associated with $u$), where the constants $\delta, C > 0$ depend only on $n$ and a pinching for the quantity $\det(u_{ab})$, that is, on constants $c'_1, c'_2 > 0$ such that
\begin{equation*}
0 < c'_1 < \det(u_{ab}) < c'_2.
\end{equation*}
In the proof of Lemma \ref{detEstimate} we have shown that our current assumptions on $e^{\ii\hth}$ and $F$ imply a uniform $C^0$ bound for the function $\tilde{A}$, depending only on $\tan(\hth)$, $\alpha$, $\sup |A|$, see \eqref{aprioriPhase}. Moreover Lemma \ref{detEstimate} shows that the pinching constants $c'_1, c'_2 > 0$ can be chosen uniformly, depending only on the same quantities. Recalling that $w = (\det(u_{ab}))^{-1}$, we find that there is a uniform $C^{\delta}$ bound on $w$ on the ball $B_{1/2}(0)$, in terms of $\tan(\hth)$, $\alpha$, $\sup |A|$. With our current conventions, the ball $B_{1/2}(0)$ does not contain a period domain $\Omega = [0,1]\times[0,1]$, but this is only a matter of notation. For example we could have started with the lattice $\Lambda = \frac{1}{4}\Z^2\oplus\frac{\sqrt{-1}}{4}\Z^2$. So we get a uniform a priori $C^{\delta}$ bound on $w$ everywhere. 

By the definition of $w$, and the regularity we just obtained, the function $u$ satisfies the Monge-Amp\`ere equation
\begin{equation*}
w(x) \det(u_{ab}(x)) = 1,
\end{equation*}
with $C^{\delta}$ coefficients. A well-known Schauder estimate due to Caffarelli \cite{caff} shows that then there is a uniform a priori $C^{2,\delta}$ bound on $u(x)$, depending only on $\sup|A|$, $\alpha$, $e^{\ii\hth}$. 

We claim that this implies a uniform bound $\Lambda_0 I < u_{ij} < \Lambda_1 I$. Equivalently, we need to show that the eigenvalues of the Hessian $u_{ij}$ are uniformly bounded, and bounded away from $0$, in terms of the usual quantities. But this follows immediately from the uniform bound on the determinant $0 < c'_1 < \det(u_{ab}) < c'_2$ and the uniform bound on $||u||_{C^2}$, which we established above.
\end{proof}
\begin{rmk}\label{tanRmk} The proof and \eqref{aprioriPhase} actually show that the bounds only depend on an upper bound for the quantities $\alpha$, $\alpha |\tan(\hth)|$ and $\sup|A|$. 
\end{rmk}
We can now complete the proof of Proposition \ref{aprioriPropSurfaces}. Recall that this claims that there are a priori bounds of all orders on solutions $(g, F)$ of \eqref{coupled_dHYM_torus} and on the positivity of the metric $g$, depending only on $e^{\ii\hth}$, $\alpha$, $f$, under the conditions
\begin{equation*}
\sin(\hth) < 0,\, \cos(\hth) > 0,\, F\geq 0.
\end{equation*}
It is convenient to write 
\begin{equation*}
F_{ij} = \left[\frac{1}{2} y^T B y + \varphi(y)\right]_{ij}
\end{equation*} 
where $B$ is a fixed, symmetric positive semidefinite matrix, and $\varphi(x)$ is periodic, with period domain $\Omega$, and satisfies $\varphi(0)=0$. 

We will use the well-known reduction of the dHYM equation to complex Monge-Amp\`ere equation in the case of a complex surface (see e.g. \cite{collinsXie}). As in the proof of Lemma \ref{detEstimate}, we write the dHYM equation as
\begin{equation*}
- F^2 \sin (\hth  ) + 2 F\wedge \omega  \cos (\hth  ) = -\omega^2 \sin (\hth  ).
\end{equation*}
We consider a general equation of the form
\begin{equation*}
(c_1 F + c_2 \omega)^2 = c_3 \omega^2 
\end{equation*}
for some choice of constants $c_i$. This is of course
\begin{equation*}
 c^2_1 F^2 + 2c_1 c_2 F\wedge \omega = (c_3 - c^2_2) \omega^2. 
\end{equation*}
Recall we have $\sin(\hth ) < 0$, $\cos(\hth ) > 0$. Then choosing
\begin{equation*}
c_1 =  (-\sin(\hth ))^{1/2},\,c_2 = \frac{\cos(\hth )}{(-\sin(\hth ))^{1/2}},\,c_3= -\frac{1}{\sin(\hth )}
\end{equation*}
shows that the dHYM condition becomes the complex Monge-Amp\`ere equation
\begin{equation*}
\left((-\sin(\hth ))^{1/2} F - \frac{\cos(\hth )}{(-\sin(\hth ))^{1/2}} \omega\right)^2 = -\frac{1}{\sin(\hth )}\omega^2 
\end{equation*}
or equivalently
\begin{equation*}
\chi^2 = \omega^2 
\end{equation*}
where 
\begin{equation*}
\chi  = -\sin(\hth ) F + \cos(\hth ) \omega .   
\end{equation*}
We should think of this as an equation for $\chi$, and so $F$, given $\omega$. Note that $\chi$ is automatically a K\"ahler class. By the Calabi-Yau Theorem, this Monge-Amp\`ere equation is solvable iff
\begin{equation*}
\int \chi^2 =  \int\omega^2, 
\end{equation*}
which of course determines $e^{\ii \hth}$ just as before. In our situation, this reduces to the real Monge-Amp\`ere
\begin{equation}\label{realMA}
\det(-\sin(\hth) F_{ij} + \cos(\hth)v_{ij}) = \det(v_{ij}).
\end{equation}
By Lemma \ref{highEstimatesLem} and the Legendre transform we have a uniform estimate on $|| v||_{C^{2,\delta}}$, depending only on $\sup|A|$, $\alpha$, $e^{\ii\hth}$. Moreover, just as in the proof of that Lemma, we observe that by \cite{fengSz} Lemma 4 we have a uniform $C^1$ bound on $\varphi$. This is independent of the equation satisfied by $\varphi$ and holds simply by periodicity, (semi)positivity $F_{ij} \geq 0$ and the normalisation $\psi(0)=0$. In particular we have a uniform $C^0$ bound on $\varphi$. Caffarelli's H\"older estimates for the real Monge-Amp\`ere equation then give a uniform bound on $||\varphi||_{C^{2, \delta}}$, depending only on $\sup|A|$, $\alpha$, $e^{\ii\hth}$. In particular we have a uniform $C^{\delta}$ bound on $F_{ij}$.

We use the latter estimate on the bundle curvature $F$ in the linearised Monge-Amp\`ere equation \eqref{linearMA}, yielding a uniform $C^{2, \delta}$ bound on $w = (\det(u_{ab}))^{-1}$ and so in turn a uniform bound on $||u||_{C^{4, \delta}}$, depending only on $|| A ||_{C^{\delta}}$, $\alpha$, $e^{\ii \hth}$.

We can now proceed inductively, using the equations \eqref{linearMA} and \eqref{realMA}, to obtain estimates of all orders on $v$ and $\varphi$, depending only on $A$, $\alpha$, $e^{\ii \hth}$. Proposition \ref{aprioriPropSurfaces} follows.\\  

Given the a priori estimates of Proposition \ref{aprioriPropSurfaces}, we are in a position to prove Theorem \ref{MainThmSurface}. Recall that this involves the choice of coupling constant 
\begin{equation*}
\alpha = \alpha' \cos(\hth)
\end{equation*}
for some fixed $\alpha' > 0$. The proof relies on the continuity method. We apply this to the family of equations, depending on a parameter $t \in [0,1]$, given by
\begin{align}\label{continuityMethodSurface} 
\begin{cases}
  \im\left(e^{-\ii\hth}\det(u^{ij} - \ii F_{ij}(\nabla u))\right) = 0\\
  [u^{ij}]_{ij}  = \tilde{A}_t -4 (1-t) \int_{\Omega} A(x)d\mu(x),  
 \end{cases}
\end{align} 
where we set
\begin{equation*}
\tilde{A}_t = -4 \alpha'\cos(\hth)  \re\left(e^{-\ii\hth}\det(u^{ij} - \ii F_{ij}(\nabla u))\right) \det(u_{ab}) -4 t A.
\end{equation*} 
For $t = 0$ the equations are solvable by choosing $v_{ij} = u^{ij}$ and $F_{ij}$ to be \emph{constant} representatives of their cohomology classes. 

By Proposition \ref{aprioriPropSurfaces} the set of times $t \in [0,1]$ for which the equations \eqref{continuityMethodSurface} are solvable is closed as long as the solution satisfies $F \geq 0$. We claim that if the classes $[F]$ and $[\omega]$ are sufficiently close, depending only on $\sup|A|$ and $\alpha'$, then the bundle curvature actually remains negative, i.e. the condition $F > 0$ is \emph{closed} along the continuity path. To see this we use the complex Monge-Amp\`ere equation 
\begin{equation*}
(-\sin(\hth) F + \cos(\hth)\omega)^2 = \omega^2
\end{equation*}
satisfied by $F$. We have shown in the proof of Lemma \ref{detEstimate} that, assuming only semiposivity $F\geq 0$, one has a priori bounds on $\omega$,
\begin{equation*}
0 < \Lambda_0 \omega_0 < \omega < \Lambda_1 \omega_0,\, ||\omega||_{C^{\delta}(\omega_0)} < \Lambda_{2,\delta}, 
\end{equation*}
depending only on an upper bound for the quantities $\alpha$, $\alpha |\tan(\hth)|$, $\sup|A|$ (see Remark \ref{tanRmk}). Thus, with our choice of coupling constant $\alpha = \alpha'  \cos(\hth)$ for some fixed $\alpha'$, we see that the constants $\Lambda_0$, $\Lambda_1$, $\Lambda_{2,\delta}$ above can be chosen uniformly in terms of $\alpha'$, $\sup|A|$ only, and in particular do not depend on $t \in [0, 1]$ and on $[F]$. Then with our assumptions, if $[F] - [\omega] \to 0 \in H^{1,1}(X, \R)$ we have $\cos(\hth) \to 0^+$, $-\sin(\hth) \to 1^-$ and $|| F - \omega||_{\omega} \to 0$. So by choosing $[F]$ and $[\omega]$ to be sufficiently close, depending only on $\alpha'$, $\sup|A|$ we can make sure that $F$ remains strictly positive along the continuity path.

It remains to check openness of the continuity path. The condition $F>0$ is clearly open, so we only need to show that, with our choice $\alpha = \alpha'\cos(\hth)$, the linearisation of the equations \eqref{continuityMethodSurface} at any point of the path are always solvable provided $[F]$, $[\omega]$ are sufficiently close, in a uniform way. Consider the equations obtained in the limiting case $\cos(\hth) = 0$:
\begin{align*}
 \begin{cases}
  F^2 = \omega^2 \\
  s(\omega) = f(\nabla u). 
  \end{cases}
\end{align*}
By the results of \cite{fengSz} Section 2, the corresponding linearised equations are uniquely solvable, so the same holds for the linearisation of \eqref{continuityMethodSurface} for $\cos(\hth)$ sufficiently small, depending only on our a priori estimates on solutions of \eqref{continuityMethodSurface}, and so on $\alpha'$, $A$. This completes the proof of Theorem \ref{MainThmSurface}.

\section{Large/small radius limits and Theorem \ref{MainThmLargeRadius}}\label{KYMSec}

We begin with a proof of Proposition \ref{LargeRadiusProp}. As usual, it is convenient to write $F(h) = \ii F$ and set $z = \int_X (t\omega - \ii F)^n$. Identifying top classes with their integrals, we may expand as $t \to \infty$
\begin{equation*}
z = \big(t^n [\omega]^n - t^{n-2} \frac{n(n-1)}{2}[\omega]^{n-2} \cup [F] \cup [F]\big) -  \ii t^{n-1} n [\omega]^{n-1} \cup [F] + O(t^{n-3}).
\end{equation*} 
By definition, we have
\begin{align*}
e^{-\ii \hth} &= \frac{\bar{z}}{|z|}\\
&= \left(1 - \frac{(n [\omega]^{n-1} \cup [F])^2}{2([\omega]^n)^2}\frac{1}{t^2}\right) + \ii \frac{n [\omega]^{n-1} \cup [F]}{[\omega]^n}\frac{1}{t} + O\left(\frac{1}{t^3}\right).
\end{align*} 
It follows that we have
\begin{equation*}
\im\left(e^{-\ii\hth}(t\omega - \ii F)^n\right) = n\left(- \omega^{n-1}\wedge F + \frac{[\omega]^{n-1} \cup [F]}{[\omega]^n}\omega^n\right)t^{n-1} + O(t^{n-3}), 
\end{equation*}
and similarly
\begin{align*}
&\re\left(e^{-\ii\hth}(t\omega - \ii F)^n\right)\\& =  \omega^n t^n\\
& -\left(\frac{n(n-1)}{2} \omega^{n-2} \wedge F \wedge F - \frac{n [\omega]^{n-1} \cup [F]}{[\omega]^n} n\omega^{n-1}\wedge F + \frac{(n [\omega]^{n-1} \cup [F])^2}{2([\omega]^n)^2}\omega^n\right )t^{n-2}\\& + O(t^{n-4}). 
\end{align*}
Now Proposition \ref{LargeRadiusProp} follows at once from the definition of $\mu_{\widetilde{\G}}$.\\

The case of small radius, Proposition \ref{SmallRadiusProp}, is similar. We have, as $t \to 0$
\begin{equation*}
z = (-\ii)^n\big([F]^n + \ii n [\omega]\cup [F]^{n-1} t + O(t^2)\big),
\end{equation*}
so
\begin{align*}
e^{-\ii \hth} &= \frac{\bar{z}}{|z|}\\
&= (\ii)^n \left(1 - \ii \frac{n [\omega]\cup [F]^{n-1}}{[F]^n}t + O(t^2)\right),
\end{align*}
and
\begin{equation*}
e^{-\ii \hth} (t\omega - \ii F)^n = F^n + \ii \left(n\omega\wedge F^{n-1}-\frac{n[\omega]\cup[F]^{n-1}}{[F]^n} F^n\right)t +O(t^2).
\end{equation*}
Proposition \ref{SmallRadiusProp} follows immediately.\\

We proceed to prove our main results concerning the K\"ahler-Yang-Mills system on an abelian variety $X$ of arbitrary dimension $n$, Proposition \ref{aprioriProp} and Theorem \ref{MainThmLargeRadius}. Recall the system is given by 
\begin{align*}
\begin{cases}
 \Lambda_g F = \mu\\
  s(g) + \alpha \Lambda^2_g ( F \wedge F ) = f, 
  \end{cases}
\end{align*}
where $F(h) = \ii F$ is the curvature of a Hermitian metric on the fibres of a holomorphic line bundle $L\to X$, of degree $-\mu$, and $f \in C^{\infty}(M)$ is a prescribed function. Note that a solution $g$ must satisfy the cohomological constraint
\begin{equation}\label{constrKYM}
\int_M f \frac{\omega^n_g}{n!} = \alpha \int_M F  \wedge F  \wedge \frac{\omega^{n-2}_g}{(n-2)!}.
\end{equation}
Using the identity (involving pointwise norms)
\begin{equation*}
\Lambda^2_g ( F \wedge F ) = 2||\Lambda_g F ||^2_g - 2||F ||^2_g
\end{equation*}
the equations can be written in the form
\begin{align*}
\begin{cases}
  \Lambda_g F  = \mu,\\
  s(g) = f + 2\alpha || F ||^2_g - 2 \alpha\mu^2.
  \end{cases}
\end{align*}
Following \cite{fengSz} we may assume, without loss of generality, that $X$ is the abelian variety $\C^n/(\Z^n\oplus \sqrt{-1}\Z^n)$ and $[\omega_0]$ is the class of the metric $\omega_0 =  \sqrt{-1} \sum_i dz_i \wedge \overline{dz}_i$. Then $U(1)^n$ acts on $X$, by translations in the $\ii\R^n$ direction. We suppose that $f$ is $U(1)^n$-invariant and look for $U(1)^n$-invariant solutions. As in the previous Section, we formulate the problem in terms of the convex function
\begin{equation*}
v(y) = \frac{1}{2}|y|^2 + \psi(y)
\end{equation*}
where $\psi\!:\R^n \to \R$ is periodic, with period domain $\Omega = [0,1]^n$, and of its Legendre transform
\begin{equation}\label{potential}
u(x) = \frac{1}{2}|x|^2 + \phi(x).
\end{equation}   
where $\phi\!:\R^n \to \R$ is periodic with the same period domain $\Omega$. As before, the invariant metric $g$ is given by the real Hessian of $v(y)$, namely $\omega_g = \sqrt{-1} \sum_{i,j}  v_{ij}dz_i\wedge \overline{dz}_j$. Then we have
\begin{align*}
& s(g) = -\frac{1}{4} v^{ij}[\log\det(v_{ab})]_{ij},\\
& ||F ||^2_{g} = v^{ij}v^{kl} F_{il} F_{kj}.
\end{align*} 
So the equations become
\begin{align}\label{cplxEqu}
\begin{cases}
  v^{ij} F_{ij} = \mu\\
  v^{ij}[\log\det(v_{ab})]_{ij} = -4 f + 8\alpha \mu^2 - 8\alpha v^{ij}v^{kl} F_{il} F_{kj}.
\end{cases}
\end{align} 
Here $F = F_{ij}$ is regarded as a periodic function with values in symmetric matrices (so we have $F(h) =  -\sum_{i,j} F_{ij} dz_{i}\wedge \overline{dz}_j$). 

In terms of the Legendre transform, the equations are
\begin{align}\label{symplEqu}
\begin{cases}
  u_{ij} F_{ij}(\nabla u) = \mu\\
  [u^{ij}]_{ij} = -4 A + 8\alpha \mu^2 - 8\alpha u_{ij}u_{kl} F_{il}(\nabla u) F_{kj}(\nabla u),
\end{cases}
\end{align} 
where we set $A(x) = f(y) = f(\nabla u)$. From this symplectic viewpoint, it is trivial to take the constraint \eqref{constrKYM} into account: the datum $A(x)$ must simply satisfy
\begin{equation*}
\int_{\Omega} A(x) d\mu(x) = \alpha \int_M F \wedge F \wedge \frac{\omega^{n-2}_g}{(n-2)!} 
\end{equation*}
with respect to the fixed Lebesgue measure $d\mu(x)$. As in the previous Section, we start from the equations is symplectic coordinates with datum $A(x)$ and \emph{define} $f(y)$ through the relation $f(\nabla u(x)) = A(x)$.\\

We can now establish our a priori estimates, Proposition \ref{aprioriProp}. A simple computation shows that we have
\begin{align*}
& v^{ij} F_{ij} = \tr(\Hes(v)^{-1} F),\\
& v^{ij}v^{kl} F_{il} F_{kj} = \tr\left((\Hes(v)^{-1} F)^2\right).
\end{align*}
So our equations, in invariant complex coordinates, are equivalent to
\begin{align*}
\begin{cases}
  \tr(\Hes(v)^{-1} F) = \mu,\\
  v^{ij}[\log\det(v_{ab})]_{ij} = -4 f + 8\alpha \mu^2 - 8\alpha \tr\left((\Hes(v)^{-1} F)^2\right).
\end{cases}
\end{align*}
Suppose $v$, $F$ give a solution with $F \geq 0$ (so in particular the bundle $L$ is seminegative). Note that $\Hes(v)^{-1} F$ is a product of symmetric matrices and so it is similar to a symmetric matrix, hence it has real eigenvalues $\lambda_i$. Since both $\Hes(v)^{-1}$ and $F$ are positive definite and semidefinite respectively, by assumption, we have in fact $\lambda_i \geq 0$.
The condition
\begin{equation*}
\tr(\Hes(v)^{-1} F) = \mu > 0
\end{equation*}
immediately gives the bound $0 \leq \lambda_i \leq \mu$. Therefore
\begin{equation*}
0 \leq \tr\left((\Hes(v)^{-1} F)^2\right) = \sum_{i} \lambda^2_i < n\mu^2.
\end{equation*}
It follows immediately that under the semipositivity assumption $F\geq 0$ there is a uniform $C^0$ bound for the image under Legendre duality of the quantity 
\begin{equation}\label{tildeA}
\tilde{A} =  -4A  + 8\alpha \mu^2 - 8\alpha u_{ij}u_{kl} F_{il}(\nabla u) F_{kj}(\nabla u),
\end{equation}
depending only on $\sup |f|$,$n$,$\mu$. This bound is preserved under pullback by the diffeomorphism induced by Legendre duality. On the other hand the second equation in \eqref{symplEqu} is precisely $[u^{ij}]_{ij} = \tilde{A}$. If follows that we have a uniform a priori bound on the quantity $[u^{ij}]_{ij}$, depending only on $\sup |f|$, $n$, $\mu$. From here, proceeding exactly as in the proof of Lemma \ref{highEstimatesLem}, we find that there are uniform constants $0 < \Lambda_0 < \Lambda_1$, $0 < \delta < 1$, $\Lambda_{2,\delta} > 0$, depending only on $A$, $n$, $\mu$, such that for a solution $u$, $F$ of \eqref{symplEqu} with $F\geq 0$ we have
\begin{align}\label{metricBoundsKYM}
\nonumber&\Lambda_0 I < u_{ij} < \Lambda_1 I,\\
&|| u ||_{C^{2,\delta}} < \Lambda_{2,\delta}.
\end{align}

Let us now consider the bundle curvature, or equivalently the form $F$. Recall that, on the universal cover, this is given by the Hessian of a function,  
\begin{equation*}
F_{ij} = \left[\frac{1}{2} y^T B y + \varphi(y)\right]_{ij},
\end{equation*}
where $B$ is a fixed, symmetric positive semidefinite matrix, and $\varphi(x)$ is periodic, with period domain $\Omega$. We can normalise $\varphi$ so that $\varphi(0)=0$. The HYM equation $\Lambda_g F = \mu$ satisfied by $h$ can be seen as a second order linear elliptic PDE, with periodic coefficients, satisfied by the periodic function $\varphi$,
\begin{equation}\label{poisson}
v^{ij}\varphi_{ij} = \mu - v^{ij} B_{ij}. 
\end{equation}
By standard Schauder theory and periodicity there is a bound
\begin{equation*}
|| \varphi||_{C^{2,\delta}(\Omega)} \leq C_{2, \delta} \left(||\varphi||_{C^{0}(\Omega)} + ||\mu - v^{ij}B_{ij}||_{C^{2,\delta}(\Omega)}\right),
\end{equation*} 
where $C_{2, \delta} > 0$ depends only on $|| v^{ij} ||_{C^{2,\delta}(\Omega)}$ and the ellipticity constants. By our previous a priori bounds \eqref{metricBoundsKYM} and the Legendre transform, both quantities are uniformly bounded in terms of $\sup|f|$, $n$, $\mu$. It follows that in fact we have an a priori bound of the form
\begin{equation*}
|| \varphi||_{C^{2,\delta}(\Omega)} \leq C_{2, \delta}  ||\varphi||_{C^{0}(\Omega)} + C'_{2, \delta}
\end{equation*} 
where the constants $C_{2, \delta}, C'_{2, \delta}>0$ depend only on $\sup|f|$, $n$, $\mu$. Moreover, just as in the proof of Lemma \ref{highEstimatesLem} and by \cite{fengSz} Lemma 4, we have a uniform $C^1$ bound on $\varphi$, so we see that $|| \varphi||_{C^{2,\delta}(\Omega)}$ is controlled only by $\sup|f|$, $n$, $\mu$.

As in the proof of Proposition \ref{aprioriPropSurfaces}, we can use this estimate on $\varphi$ in the linearised Monge-Amp\`ere equation \eqref{linearMA}, with $\tilde{A}$ now given by \eqref{tildeA}. The resulting $C^{\delta}$ bound on $\tilde{A}$ gives a $C^{2, \delta}$ bound on $w = (\det(u_{ab}))^{-1}$ and so a uniform bound on $||u||_{C^{4, \delta}}$, depending only on $|| A ||_{C^{\delta}}$, $\mu$, $n$.

We can now proceed inductively, using the linearised Monge-Amp\`ere \eqref{linearMA} (with right hand side given by \eqref{tildeA}) and the Poisson equation \eqref{poisson}, to obtain estimates of all orders on $v$ and $\varphi$, depending only on $A$, $\mu$, $n$. Proposition \ref{aprioriProp} follows.\\  

We are in a position to prove Theorem \ref{MainThmLargeRadius}. Recall for this result we have a negative line bundle $L$, respectively an ample line bundle $N$ on $X$, where $L$ has degree $-\mu$ and $[\omega] = c_1(N)$. We are concerned with the system
\begin{align*}
\begin{cases} 
  u_{ij} F_{ij}(\nabla u) = k \mu\\
  [u^{ij}]_{ij} = -4A + 8\frac{\alpha}{k^2} \mu^2 - 8\frac{\alpha}{k^2} u_{ij}u_{kl} F_{il}(\nabla u) F_{kj}(\nabla u).
\end{cases}
\end{align*} 
Here $\ii F$ is the curvature of a metric on the fibres of $L\otimes (k-1)\mu N$ for some $k \geq 1$. We claim that by taking $k$ sufficiently large, depending only on $\sup|A|$, $n$, $\mu$, we can find $u$ and $F$ solving the equations.

For the proof it is convenient to work instead with the $\Q$-line bundle $(1-\beta)L + \frac{\beta}{n} N$. Here $\beta$ is a parameter in the construction, to be chosen appropriately. At the end of the argument we will see how to obtain from this a solution on a genuine line bundle. So we have
\begin{equation*}
F_{ij} = (1-\beta) \tilde{F}_{ij} + \beta\frac{\mu}{n} v_{ij},\, \beta \in (0, 1)\cap \Q
\end{equation*}  
where $\ii \tilde{F}$ is the curvature of some metric on the fibres of $L$. 
 
The proof relies on the continuity method. We apply this to the family of equations, depending on the parameter $\beta$, given by
\begin{align}\label{ContPathSymp}
\begin{cases}
  u_{ij} F_{ij}(\nabla u) = \mu\\
  [u^{ij}]_{ij} = -4tA - 4(1-t) \int_{\Omega}A(x)d\mu(x)+ 8\alpha \mu^2 - 4 u_{ij}u_{kl} F_{il}(\nabla u) F_{kj}(\nabla u)\\
 F > 0,
\end{cases}
\end{align}
for $t \in [0,1]$. 

When $t = 0$ a solution of \eqref{ContPathSymp} is given in complex coordinates by taking 
\begin{align*}
 v(y) = \frac{1}{2} |y|^2,\,F_{ij} = (1-\beta) B_{ij} + \beta\frac{\mu}{n} v_{ij} > 0
\end{align*}
for all $\beta \in (0,1)\cap\Q$, where $B > 0$ is a constant symmetric matrix. 

We will show in the Appendix that the linearised equations corresponding to the system \eqref{symplEqu} are uniquely solvable. This implies that the set of times $t \in [0, 1]$ for which \eqref{ContPathSymp} has a solution is open.  

We claim that this set is also closed. By Proposition \ref{aprioriProp} we have a priori $C^{k}$ estimates of all orders on solutions of \eqref{ContPathSymp}, as well as on the positivity of the solution metric $g$, which only depend on $A$, $n$, $\mu$, and in particular are independent of $t \in [0,1]$. Then it follows from the Ascoli-Arzel\`a Theorem that we can take the limit of a (sub)sequence of solutions, corresponding to times $t_i \to t \in [0,1]$ and obtain a solution at time $t$. 

It remains to be seen that the positivity condition $F > 0$ is also closed. Arguing by contradiction we assume that for the limit solution we have $F \geq 0$ but not $F > 0$. So we have 
\begin{equation*}
(1-\beta) \tilde{F}_{ij}(\xi) + \beta\frac{\mu}{n} v_{ij}(\xi) = 0
\end{equation*}
for some unit vector $\xi$. Recall that $\tilde{F}$ takes the form 
\begin{equation*}
\tilde{F}_{ij} = \left[\frac{1}{2} y^T B y + \varphi(y)\right]_{ij}
\end{equation*}
for a periodic function $\varphi$, so we find
\begin{equation}\label{blowupHess}
\Hes(\varphi)(\xi) = - B(\xi) - \frac{\beta}{1-\beta} \frac{\mu}{n} \Hes{v}(\xi). 
\end{equation}
The equation satisfied by $F_{ij} = (1-\beta) \tilde{F}_{ij} + \beta\frac{\mu}{n} v_{ij}$ in complex coordinates is
\begin{equation*}
v^{ij} \big((1-\beta) \tilde{F}_{ij} + \beta\frac{\mu}{n} v_{ij}\big) = \mu. 
\end{equation*}
It follows that $\tilde{F}$ satisfies the equation
\begin{equation*}
v^{ij} \tilde{F}_{ij} = \mu, 
\end{equation*}
or is terms of $\varphi$
\begin{equation*}
v^{ij} \varphi_{ij} = \mu - v^{ij}B_{ij}. 
\end{equation*}
Recall we are free to normalise $\varphi$ with an additive constant. In particular we can assume that $\varphi$ is $L^2$-orthogonal with respect to the metric $g$ to constant functions, that is, to the kernel of the Laplacian $\Delta_g$. With this assumption we have a standard Schauder estimate
\begin{equation*}
|| \varphi||_{C^{2,\delta}(\Omega)} \leq K_{2, \delta} ||\mu - v^{ij}B_{ij}||_{C^{\delta}(\Omega)},
\end{equation*}
where $K_{k, \delta} > 0$ depends only on $|| v^{ij} ||_{C^{\delta}(\Omega)}$ and the ellipticity constants. By the proof of Proposition \ref{aprioriProp} all these quantities,  with $||\mu - v^{ij}B_{ij}||_{C^{\delta}(\Omega)}$, are uniformly bounded in terms of $\sup|A|$, $n$, $\mu$, assuming only $F \geq 0$ (in particular, independently of $t\in [0,1]$ and $\beta \in (0,1)\cap \Q$). It follows that we have a uniform bound on $||\varphi||_{C^2(\Omega)}$. But \eqref{blowupHess}, together with the strictly positive uniform lower bound on $\Hes(v)$ given by Proposition \ref{aprioriProp}, implies that $\Hes(\varphi)(\xi)$ can be made arbitrarily large by taking $\beta$ sufficiently close to $1$. This is a contradiction, arising from our assumption that at some time $F$ is only semipositive.

The upshot is that for all rational $\beta$ sufficiently close to $1$, depending only on $\sup|A|$, $n$, $\mu$, we have $g$, $F$, providing a solution on the $\Q$-line bundle $(1-\beta) L + \beta \frac{\mu}{n} N$. In general, if $g$, $F$ give a solution to the K\"ahler-Yang-Mills equations \eqref{KYMIntro} on a line bundle, then the same $g$ together with the rescaled $2$-form $\rho F$ solve the system
\begin{align*} 
\begin{cases}
  \Lambda_g (\rho F ) = \rho \mu I\\
  s(g) + \frac{\alpha}{\rho^2} \Lambda^2_g ( (\rho F ) \wedge (\rho F ) ) = \lambda.
\end{cases}
\end{align*}  
We apply this simple consideration to our solution $g$, $F$ above, choosing $\beta = 1 - \frac{1}{k}$ for sufficiently large $k$, and with scaling factor $\rho = n k$. This yields a solution defined on the fibres of $L^n \otimes (k-1)\mu N$, and with parameter $\frac{\alpha}{(n k)^2}$, as claimed by Theorem \ref{MainThmLargeRadius}.
\section{Proof of Theorems \ref{MainThmODE} and \ref{LargeSmallRadiusThmODE} }\label{ODESec}
In the present context $X$ is the abelian surface $\C^2/(\Z^2\oplus\ii\Z^2)$ and the datum $f$ only depends on a single coordinate, say $y_1$. Equivalently, its Legendre transform $A$ only depends on $x_1$.
The coupled equations \eqref{coupled_dHYM} are equivalent to the system
\begin{align*}
\begin{cases}
 \im\left(e^{-\ii\hth}\det(v_{ij} - \ii F_{ij})\right) = 0\\
 -\frac{1}{4} v^{ij}[\log\det(v_{ab})]_{ij} -\alpha \frac{\re\left(e^{-\ii\hth}\det(v_{ij} - \ii F_{ij})\right)}{\det(v_{ab})} = f(y_1),  
\end{cases}
\end{align*}
to be solved for matrices $v$, $F$ of the form
\begin{equation*}
v_{ij} = \left(\begin{matrix}
1 + \psi''(y_1) & 0\\
0 & 1
\end{matrix}\right),\, F_{ij} = F^0_{ij} + \varphi''(y_1)\delta_{1i}\delta_{1j} = \left(\begin{matrix}
a + \varphi''(y_1) & b\\
b & c
\end{matrix}\right). 
\end{equation*}
Here we have, a priori, $1+ \psi''(y_1) > 0$ , while we are not imposing positivity conditions on $F_{ij}$. The phase $e^{\ii\hth}$ is determined by the constraint 
\begin{equation*}
\int_{\Omega} \det(v_{ij} - \ii F_{ij}) \in \R_{>0} e^{\ii\hth}.
\end{equation*}
We have
\begin{equation*}
\det(v_{ij} - \ii F_{ij}) = \big(1 -\det(F^0) + \psi'' - c\varphi''\big)-\ii \big(\tr(F^0) + \varphi'' + c\psi''\big),
\end{equation*}
so integrating using the periodic boundary conditions shows 
\begin{equation*}
e^{\ii\hth} = \frac{1 -\det(F^0) - \ii  \tr(F^0) }{\big(\big(1 -\det(F^0)\big)^2 + \big(\tr(F^0)\big)^2\big)^{1/2}}.
\end{equation*}
Similarly we have
\begin{align*}
&\im\left(e^{-\ii\hth}\det(v_{ij} - \ii F_{ij})\right)\\
& =  -\sin (\hth) \left(1-\det(F^0) - c \varphi''+\psi'' \right)-\cos (\hth) \left(\tr(F^0) + c \psi'' +\varphi ''\right),
\end{align*}
so that the dHYM equation becomes the algebraic identity
\begin{equation*}
\varphi'' = -\frac{\sin (\hth) \left(1 - \det(F^0)\right)+ \cos (\hth) \tr(F^0)+\psi ''(x) (c \cos (\hth)+\sin (\hth))}{\cos (\hth)-c \sin (\hth)}.
\end{equation*}
Using this identity for $\varphi''$ gives
\begin{align*}
&\frac{\re\left(e^{-\ii\hth}\det(v_{ij} - \ii F_{ij})\right)}{\det(v_{ab})}\\&=\frac{b^2}{\left(\psi '' +1\right) (\cos (\hth )-c \sin (\hth ))}-\frac{c^2+1}{c \sin (\hth  )-\cos (\hth )}.
\end{align*}
On the other hand the scalar curvature is given by
\begin{equation*}
-\frac{1}{4} v^{ij}[\log\det(v_{ab})]_{ij} = -\frac{1}{4}\frac{(\log(1+\psi''))''}{1+\psi''},
\end{equation*}
so that the coupled equations \eqref{coupled_dHYM} become the single nonlinear ODE
\begin{align*}
-\frac{1}{4}\frac{(\log(1+\psi''))''}{1+\psi''} -\frac{\alpha b^2}{\left(\psi '' +1\right) (\cos (\hth )-c \sin (\hth ))} +\frac{\alpha(c^2+1)}{c \sin (\hth  )-\cos (\hth )} = f(y_1).
\end{align*}
Now consider the Legendre transform of the convex function \emph{of a single variable}
\begin{equation*}
\frac{1}{2}y^2_1 + \psi(y_1).
\end{equation*}
This takes the form
\begin{equation*}
\frac{1}{2}x^2_1 + \phi(x_1) 
\end{equation*}
for some periodic $\phi(x_1)$, and by the standard Legendre property
\begin{equation*}
1 + \phi''(x_1) = \frac{1}{1 + \psi''(y_1)} 
\end{equation*}
together with the one-dimensional case of Abreu's formula for the scalar curvature
\begin{equation*}
-\frac{1}{4}\frac{(\log(1+\psi(y_1)''))''}{1+\psi(y_1)''} = -\frac{1}{4}\left(\frac{1}{1+\phi''(x_1)}\right)''
\end{equation*}
we find that the above nonlinear ODE, to which we reduced the coupled equations \eqref{coupled_dHYM}, can be written in terms of $\phi(x_1)$ as
\begin{equation}\label{coupled_dHYM_ODE}
-\frac{1}{4}\left(\frac{1}{1+\phi''}\right)'' -\frac{\alpha b^2(1 + \phi'')}{ \cos (\hth )-c \sin (\hth )} +\frac{\alpha(c^2+1)}{c \sin (\hth  )-\cos (\hth )} - A(x_1)=0,
\end{equation}
where $A(x_1)$ denotes the image of $f(y_1)$ under the Legendre transform diffeomorphism, as usual. We are assuming of course the cohomological compatibility condition 
\begin{equation*}
\int^1_0 A(x_1)dx_1 = -\frac{\alpha b^2}{ \cos (\hth )-c \sin (\hth )} -\frac{\alpha(c^2+1)}{c \sin (\hth  )-\cos (\hth )}.
\end{equation*}
In order to prove the existence of a periodic solution $\phi$, satisfying $1+\phi'' >0$, we argue precisely as in the proofs of Theorems \ref{MainThmSurface} and \ref{MainThmLargeRadius}, relying on the same continuity method and the results of Feng-Sz\'ekelyhidi. Thus in order to obtain closedness it is enough to prove that a periodic solution $\phi$ of \eqref{coupled_dHYM_ODE}, with $1+\phi''>0$, would satisfy a priori a uniform $C^0$ bound on the scalar curvature. Equivalently it is enough to prove a uniform a priori $C^0$ bound for the quantity
\begin{equation*}
\left|\frac{\alpha b^2(1 + \phi'')}{ \cos (\hth )-c \sin (\hth )}\right|
\end{equation*}  
for a solution of \eqref{coupled_dHYM_ODE}. But since we have $1 + \phi'' > 0$ by assumption, we only need to show that there is a uniform a priori bound from above, $\phi'' < C$. Thus suppose $\bar{x}_1$ is a point at which $1+\phi''(x_1)$ attains its maximum. Then at $\bar{x}_1$ the quantity $(1+\phi''(x_1))^{-1}$ attains its minimum, so we have
\begin{equation*}
\left(\frac{1}{1+\phi''(\bar{x}_1)}\right)'' \geq 0.
\end{equation*}
Using the equation \eqref{coupled_dHYM_ODE} this shows
\begin{equation*} 
\frac{\alpha b^2(1 + \phi''(\bar{x}_1))}{ \cos (\hth )-c \sin (\hth )} -\frac{\alpha(c^2+1)}{c \sin (\hth  )-\cos (\hth )} \leq  -A(\bar{x}_1) \leq \sup |A|.
\end{equation*}
Since we already know $1 + \phi''(\bar{x}_1) > 0$, and we have $\alpha > 0$, if we further assume the condition
\begin{equation*} 
\frac{b^2}{ \cos (\hth )-c \sin (\hth )} > 0
\end{equation*}
the above inequality immediately gives the required uniform bound $\phi'' < C$. But a little computation shows that we have in fact 
\begin{equation*}
\frac{\alpha b^2}{ \cos (\hth )-c \sin (\hth )} = \frac{b^2} {1+b^2+c^2}\left((1-\det(F^0))^2+(\tr( F^0))^2\right)^{1/2}
\end{equation*}
so this quantity is nonnegative, and only vanishes for $b = 0$, in which case \eqref{coupled_dHYM_ODE} reduces to the (solvable) Abreu equation.

To obtain openness for the continuity method, we need to show that the operator given by the left hand side of \eqref{coupled_dHYM_ODE}, mapping  $C^{k, \alpha}_0(S^1, d\mu)$ to $C^{k-4, \alpha}_0(S^1, d\mu)$, has surjective differential at a solution. 
The differential maps $\dot{\phi}$ to $L\left(\frac{ \dot\phi''}{(1+\phi'')^2}\right)$, where we set, for any $\beta \in C^{k, \alpha}(S^1, d\mu)$, 
\begin{equation*}
L(\beta) = \frac{1}{4}\beta'' -\frac{\alpha b^2 (1+\phi'')^2\beta}{ \cos (\hth )-c \sin (\hth )}.
\end{equation*}  
The operator $L$ acting on $C^{k, \alpha}(S^1, d\mu)$ is formally self-adjoint and has trivial kernel by the condition $\frac{\alpha b^2}{ \cos (\hth )-c \sin (\hth )} > 0$. Thus the equation $L(\beta) = \gamma$ is uniquely solvable for all periodic $\gamma$, and if $\gamma \in C^{k-2, \alpha}_0(S^1, d\mu)$ the unique solution $\beta \in C^{k, \alpha}(S^1, d\mu)$ satisfies $\int (1+\phi'')^2\beta = 0$, so in turn the equation $\frac{ \dot\phi''}{(1+\phi'')^2} = \beta$ is solvable.\\ 

Theorem \ref{LargeSmallRadiusThmODE} follows from similar arguments. The large and small radius limit equations are given by
\begin{align*}
\begin{cases}
  v^{ij} F_{ij} = \mu\\
  -\frac{1}{4} v^{ij}[\log\det(v_{ab})]_{ij} -2\alpha v^{ij}v^{kl} F_{il} F_{kj}+ 2\alpha\mu^2 = f(y_1),  
\end{cases}
\end{align*} 
respectively
\begin{align*}
\begin{cases}
 F^{ij} v_{ij} = \kappa\\
 -\frac{1}{4} v^{ij}[\log\det(v_{ab})]_{ij} -\alpha \frac{\det(F_{ab})}{\det(v_{ab})} = f(y_1),
\end{cases}
\end{align*}
where
\begin{equation*}
\mu = \tr(F^0),\,\kappa= \frac{\tr(F^0)}{\det(F^0)}.
\end{equation*}
In both cases, the HYM (i.e. Poisson) equation and the J-equation can be solved explicitly, giving the identities
\begin{equation*}
\varphi'' = a \psi'',
\end{equation*}
respectively
\begin{equation*}
\varphi'' = \frac{c \det(F^0)}{b^2 + c^2} \psi''.
\end{equation*}
Using these identities, we find that the large radius limit becomes the nonlinear ODE
\begin{equation*}
-\frac{1}{4}\frac{(\log(1+\psi''))''}{1+\psi''} - 2\alpha (a^2+c^2) - \frac{4 \alpha b^2}{1+\psi ''(x)} + 2\alpha (\tr(F^0))^2 = f(y_1),
\end{equation*}
and similarly the small radius limit becomes
\begin{equation*}
-\frac{1}{4}\frac{(\log(1+\psi''))''}{1+\psi''} -\frac{\alpha b^2\det(F^0)}{\left(b^2+c^2\right) \left(1+\psi ''(x) \right)}-\frac{\alpha c^2 \det(F^0)}{b^2+c^2} = f(y_1).
\end{equation*}
Taking the Legendre transform $\frac{1}{2}x^2_1 + \phi$ of the convex function $\frac{1}{2} y^2_1+\psi$, we obtain the ODE
\begin{equation*}
-\frac{1}{4}\left(\frac{1}{1+\phi''(x_1)}\right)'' - 2\alpha (a^2+c^2) - 4 \alpha b^2 (1+\phi ''(x))+ 2\alpha (\tr(F^0))^2 = A(y_1),
\end{equation*}
respectively
\begin{equation*}
-\frac{1}{4}\left(\frac{1}{1+\phi''(x_1)}\right)'' -\frac{\alpha b^2\det(F^0)\left(1+\phi ''(x) \right)}{\left(b^2+c^2\right)}-\frac{\alpha c^2 \det(F^0)}{b^2+c^2} = A(y_1).
\end{equation*}

We can now apply the same maximum principle argument used in the proof of Theorem \ref{MainThmODE}, to conclude that the large radius limit equations are solvable provided the condition
\begin{equation*}
\alpha b^2 > 0
\end{equation*}
is satisfied, and the same holds for the small radius limit equations under the condition
\begin{equation*}
\frac{\alpha b^2\det(F^0)}{\left(b^2+c^2\right)} > 0.
\end{equation*}
Of course the large radius limit condition holds unless $b=0$, in which case we reduce to the (solvable) Abreu equation. On the other hand the small radius limit condition $\det(F^0) > 0$ does give a nontrivial constraint.\\

Finally, let us prove Proposition \ref{HKProp}. Recall that in the proof of Theorem \ref{MainThmODE} we reduced the coupled equations \eqref{coupled_dHYM} to the single nonlinear ODE \eqref{coupled_dHYM_ODE} for the Legendre transform $1 + \phi''(x_1)$. We only need to show that if $A(x_1)$ is real analytic, then so is $1+\phi''(x_1)$. Setting $U(x_1) = (1 + \phi''(x_1))^{-1}$, \eqref{coupled_dHYM_ODE} is equivalent to the system
\begin{align*}
\begin{cases}
  U' = V,\\
  V' = -\left(\frac{4\alpha b^2}{ \cos (\hth )-c \sin (\hth )}\right)\frac{1}{U} -\frac{4\alpha(c^2+1)}{c \sin (\hth  )-\cos (\hth )} -4A(x_1).
\end{cases}
\end{align*}
If $(U, V)$ is a solution with $U > 0$ and $A(x_1)$ is real analytic, it follows from the Cauchy-Kovalevskaya Theorem that $U, V$ are also real analytic. The solution constructed in Theorem \ref{MainThmODE} satisfies $U > 0$, so Proposition \ref{HKProp} follows.

\appendix\section{Linearised equations}
This Appendix studies the linearisation of the K\"ahler-Yang-Mills equations formulated in symplectic coordinates on a torus. In particular we prove that these linearised equations correspond to a scalar linear differential operator which has trivial kernel and is formally self-adjoint, with respect to the Lebesgue measure.
 
Thus we consider the system (\ref{symplEqu}), replacing the datum $A$ with $A_t = tA +(1-t)\int_{\Omega}A(x)d\mu(x)$ for $t \in [0,1]$. In complex coordinates
\begin{equation}\label{CurvDecomposition}
F_{ij}(y) = B_{ij} + \varphi_{ij}(y),
\end{equation}
where $B_{ij}$ denotes a constant symmetric matrix; using this notation, the system (\ref{symplEqu}) in symplectic coordinates has the form:
\begin{align}\label{ContPathSymp2}
\begin{cases}
& u_{ij} B_{ij}+\partial_i u^{ij}\partial_j\varphi = \mu\\
& [u^{ij}]_{ij} + (\partial_i (u^{lm}\varphi_m))(\partial_l(u^{in}\varphi_n))\\ &+B_{ik}B_{jl}u_{ij}u_{kl}+2(\partial_j(u^{kn}\varphi_n))u_{kl}B_{jl}+A_t=0, 
\end{cases}
\end{align} 
where, with a small abuse of notation, $\varphi$ denotes also the Legendre transform of the function in (\ref{CurvDecomposition}).
In order to study the linearization of \eqref{ContPathSymp2}, we consider the linear operator $L(\dot{u})$ associated to the second equation, which has the form
\begin{equation}
\begin{split}
L(\dot{u})=&-\partial^2_{ij}(u^{ia}\dot{u}_{ab}u^{bj})+2(\partial_i(u^{lm}\dot{\varphi}(\dot{u})_m))(\partial_l(u^{in}\varphi_n))\\
& -2(\partial_i(u^{la}\dot{u}_{ab}u^{bm}\varphi_m))(\partial_l(u^{in}\varphi_n)) +2B_{ik}B_{jl}\dot{u}_{ij}u_{kl}\\
&+2(\partial_j(u^{kn}\dot{\varphi}(\dot{u})_n))u_{kl}B_{jl} +2(\partial_j(u^{kn}\varphi_n))\dot{u}_{kl}B_{jl}\\ & -2(\partial_j(u^{ka}\dot{u}_{ab}u^{bn}\varphi_n))u_{kl}B_{jl},
\end{split}
\end{equation}
where $\dot{\varphi}(\dot{u})$ is the unique solution of the linearization of the first equation in \eqref{ContPathSymp2}:
\begin{equation}\label{LinCond}
\dot{u}_{ij}B_{ij} + \Delta\dot{\varphi} -\partial_i(u^{ia}\dot{u}_{ab}u^{bj}\varphi_j) =0
\end{equation}
with the normalization $
\int_X \dot{\varphi} d\mu = 0 $ and with $\Delta = \partial_i u^{ij}\partial_j$.
In order to prove that the operator $L \colon C^{N,\alpha}_{0}(T^n) \to C^{N-4,\alpha}_{0}(T^n)$, at a solution of \eqref{ContPathSymp2}, is invertible for a sufficently large $N\in \mathbb{N}$ and $0<\alpha<1$, is enough to show that $L$ is injective and formally self-adjoint with respect to the $L^2$-product define by the volume form $d\mu$.
Then, by the implicit function theorem, if we have a smooth solution of \eqref{ContPathSymp2} for $\bar{t}\in \left[0,1\right]$, there exist $C^{N,\alpha}$ solutions for $\bar{t}+\varepsilon$, with $\varepsilon <<1$. By standard bootstraping technique, these solutions are actually smooth. 

We prove that, for any $\gamma, \xi \in C^{N,\alpha}_0(T^n)$ with $N>>1$, $\int_{X}\xi L(\gamma) =\int_{X}\gamma L(\xi)$, omitting in the following the volume form, in order to ease the notation.
We split $L(\gamma)=L_0(\gamma) + L_1(\dot{\varphi}(\gamma))$, with 
\begin{equation}\label{LinSpl}
\begin{split}
L_0(\gamma)=&-\partial^2_{ij}(u^{ia}\gamma_{ab}u^{bj})+2(\partial_j(u^{kn}\varphi_n))\gamma_{kl}B_{jl}\\
& -2(\partial_i(u^{la}\gamma_{ab}u^{bm}\varphi_m))(\partial_l(u^{in}\varphi_n)) +2B_{ik}B_{jl}\gamma_{ij}u_{kl}\\
& -2(\partial_j(u^{ka}\gamma_{ab}u^{bn}\varphi_n))u_{kl}B_{jl},\\
\end{split}
\end{equation}
\begin{equation}\label{LinSpl2}
\begin{split}
L_1(\gamma)=&+2(\partial_j(u^{kn}\dot{\varphi}(\gamma)_n))u_{kl}B_{jl} +2(\partial_i(u^{lm}\dot{\varphi}(\gamma)_m))(\partial_l(u^{in}\varphi_n)).\\
\end{split}
\end{equation}
We will show that
\begin{equation*}
\int_X  \left(\xi L_0(\gamma) - \gamma L_0(\xi)\right) = -\int_X  \left(\xi L_1(\gamma) - \gamma L_1(\xi)\right).
\end{equation*}
At a solution of $\eqref{ContPathSymp2}$, integrating by parts, we get the following identities for the different terms of $\int_X\xi L_0(\gamma)$: 
\begin{align*}
&\int \xi\partial^2_{ij}(u^{ia}\gamma_{ab}u^{bj})= \int\gamma\partial^2_{ij}(u^{ia}\xi_{ab}u^{bj});\\
&\int \xi \gamma_{kl}B_{jl}\partial_j(u^{kn}\varphi_n)=\\
&=\int \xi_{kj}u^{kn}\varphi_n\gamma_lB_{jl} + \int \xi_ku^{kn}\varphi_n\gamma_{lj}B_{lj}-\int \xi\gamma_l B_{jl}\partial_j\Delta \varphi;\\
&\int \xi(\partial_i(u^{la}\gamma_{ab}u^{bm}\varphi_m))(\partial_l(u^{in}\varphi_n)) =\\ 
&\int \xi_i u^{in}\varphi_n\partial_j(u^{ja}\gamma_{ab}u^{bl}\varphi_l) -\int \gamma_b u^{bm}\varphi_m\partial_j(u^{ja}\xi_{ab}u^{bl}\varphi_l)\\
&+\int \gamma (u^{la}\xi_{ab}u^{bm}\varphi_m)\partial_l(\Delta \varphi)-\int \xi (u^{la}\gamma_{ab}u^{bm}\varphi_m)\partial_l(\Delta \varphi)\\
&+\int \gamma\partial_a(u^{bm}\varphi_m)\partial_{b}(u^{al}\xi_{li}u^{in}\varphi_n);\\
& \int \xi B_{ik}B_{jl}\gamma_{ij}u_{kl}=\\
&\int \gamma B_{ik}B_{jl}u_{kl}\xi_{ij} +\int\gamma B_{ik}B_{jl}u_{klj}\xi_i -\int \xi B_{ik}B_{jl}u_{kli}\gamma_j;\\
&\int \xi u_{kl}B_{jl}\partial_j(u^{ka}\gamma_{ab}u^{bn}\varphi_n) = \\
&\int \gamma_bu^{bn}\varphi_n \xi_{ja}B_{ja} - \int \gamma\xi_{jb}B_{ja}\partial_a(u^{bn}\varphi_n) -\int \gamma \xi_jB_{ja}\partial_a\Delta \varphi 
\end{align*}
\begin{align*}
&+\int \xi \gamma_bu^{bn}\varphi_n u^{ka}_a u_{jkl}B_{jl} + \int \xi u^{ka}u_{jkl}B_{lj}\gamma_b \partial_a(u^{bn}\varphi_n)\\
&+\int \xi \gamma_b u^{ka}u^{bn}\varphi_n u_{jkla}B_{lj} - \int \gamma\xi_au^{ka}_b u^{bn}\varphi_n u_{jkl}B_{lj}\\
&- \int \gamma\xi_au^{ka} u_{jkl}B_{lj}\Delta \varphi   - \int \gamma u^{bn}\varphi_n \xi_a u^{ak} u_{jklb}B_{lj}\\ 
&+\int (\partial_j(u^{ka}\xi_{ab}u^{bn}\varphi_n))u_{kl}B_{lj}\gamma + \int \xi_{ab}u^{ka}u^{bn}\varphi_nu_{kl}B_{lj}\gamma_j. 
\end{align*}
Hence, after several cancellations, we get
\begin{align*}
&-\int_X \xi L_0 (\gamma) + \int_X \gamma L_0(\xi)=\\
& =\underbrace{\int \xi_iu^{in}\varphi_n\partial_j(u^{ja}\gamma_{ab}u^{bl}\varphi_l)}_{(i)} -\underbrace{\int (\gamma_bu^{bm}\varphi_m)\partial_j(u^{ja}\xi_{ab}u^{bl}\varphi_l)}_{(ii)}\\
&+ \underbrace{\int \xi_{ja}B_{ja}\gamma_bu^{bn}\varphi_n}_{(iii)} -\underbrace{\int \gamma_{ja}B_{ja}\xi_bu^{bn}\varphi_n}_{(iv)} .
\end{align*}
Consider now $(i)$ and $(iv)$; using \eqref{LinCond}, we get:
\begin{align*}
&\int \xi_iu^{in}\varphi_n(\partial_j(u^{ja}\gamma_{ab}u^{bl}\varphi_l) -\gamma_{ja}B_{ja} )\\
&= \int \xi_iu^{in}\varphi_n \Delta \dot{\varphi}(\gamma) \\
&= -\int u^{kn}\dot{\varphi}(\gamma)_n \xi_i\partial_k(u^{in}\varphi_n) -\int \dot{\varphi}(\gamma)_nu^{kn} \xi_{ik}u^{im}\varphi_m \\
&=\int \xi  (\partial_i(u^{kn}\dot{\varphi}(\gamma)_n))(\partial_k(u^{in}\varphi_n)) + \int\xi  u^{kn}\dot{\varphi}(\gamma)_n \partial_k \Delta \varphi\\
&-\int \dot{\varphi}(\gamma)_nu^{kn} \xi_{ik}u^{im}\varphi_m \\
&=\int \xi  (\partial_i(u^{kn}\dot{\varphi}(\gamma)_n))(\partial_k(u^{in}\varphi_n))+ \int \dot{\varphi}(\gamma)\partial_n (u^{nk}\xi_{ki}u^{im}\varphi_m)\\
&-\int \dot{\varphi}(\gamma)_n  u^{kn}\xi  u_{klm}B_{lm}\\
&=\int \xi  (\partial_i(u^{kn}\dot{\varphi}(\gamma)_n))(\partial_k(u^{in}\varphi_n))+ \int \dot{\varphi}(\gamma)\partial_n (u^{nk}\xi_{ki}u^{im}\varphi_m)\\
\end{align*}
\begin{align*}
&+ \int \xi  (\partial_m(u^{kn}\dot{\varphi}(\gamma)_n))u_{kl}B_{lm} + \int \xi_m \dot{\varphi}(\gamma)_nB_{nm} \\
&=\int \xi  (\partial_i(u^{kn}\dot{\varphi}(\gamma)_n))(\partial_k(u^{in}\varphi_n))+ \int \xi  u_{kl}B_{lm}\partial_m(u^{kn}\dot{\varphi}(\gamma)_n)\\
&+\int \dot{\varphi}(\gamma) \Delta \dot{\varphi}(\xi) \, .
\end{align*}
Notice that the first two terms coincide with $\int_X \xi L_1(\gamma)$, while the third one is symmetric in $\xi$ and $\gamma$. An identical computation for $(ii)$ and $(iii)$ shows that $L$ is formally self-adjoint.

A similar argument using repeated integration by parts proves that $\int_{X} \gamma L (\gamma) \leq 0$, with $\int_{X} \gamma L (\gamma)=0$ if and only if $\gamma=0$, so that $L$ has trivial kernel.

\end{document}